\documentclass[10pt]{amsart}
\usepackage{amsmath}
\usepackage{mystyle}

\begin{document}

\title{A Free Boundary Problem Related to Thermal Insulation}

\author{Luis A. Caffarelli}
\address[Luis A. Caffarelli]{Department of Mathematics, The University of Texas at Austin, Austin, Texas}
\email{caffarel@math.utexas.edu}

\author{Dennis Kriventsov}
\address[Dennis Kriventsov]{Courant Institute of Mathematical Sciences, New York University, New York}
\email{dennisk@cims.nyu.edu}  

\date{November 18, 2015}

\begin{abstract}
 We study a free boundary problem arising from the theory of thermal insulation. The outstanding feature of this set optimization problem is that the boundary of the set being optimized is not a level surface of a harmonic function, but rather a hypersurface along which a harmonic function satisfies a Robin condition. We show that minimal sets exist, satisfy uniform density estimates, and, under some geometric conditions, have ``locally flat'' boundaries. 
\end{abstract}

\maketitle

\section{Introduction}

In this article we study the following variational problem. We are given a domain $\W \ss \R^n$, and we consider the minimizing pair $(A,u)$, where $A$ is a set containing $\W$ and $u$ is a nonnegative function on $A$ with $u\equiv 1$ on $\W$, of the energy
 \begin{equation}\label{eq:model3}
  F(A,u) = \int_{A} |\n u|^2 d\cL^n + h \int_{\partial A} u^2 d\cH^{n-1}  +  C_{0} \cL^n(A).
 \end{equation}
Here $h,C_0>0$ are fixed constants. 

This problem is motivated by the following optimal insulation configuration: the variable $u$ is the temperature. In $\W$ we keep a constant temperature, $u\equiv 1$. We are able to insulate the domain $\W$ with some bulk insulation that occupies $A$, and further $A$ is covered by a thin layer $\G_\e$, of width $\e$, of highly insulating ($\approx \e^{-1}$) material. The energy of this configuration is then
 \begin{equation}\label{eq:model2}
  F_{\e}(A,u) = \int_{A} |\n u|^2 d\cL^n + h \e \int_{\G_\e} |\n u|^2 d\cL^{n}  +  C_{0} \cL^n(A),
 \end{equation}
where $C_0$ represents the per-unit cost of the insulator $A$. In the $\e$ layer $u$ will be jumping from some value $u_0$ in $A$ to $0$ in the complement of $A\cup \G_\e$, and so $|\n u| \sim -u_\nu \sim \frac{u_0}{\e}$ (here $\nu$ stands for the outward unit normal to $A$). This suggests that the limiting functional will take the form \eqref{eq:model3}.

Minimizers of \eqref{eq:model3} have very different qualitative behavior from solutions of the conventional one-phase problem (which arises when there is no layer $\G_\e$; see \cite{AC}). One possible phenomenon might occur in the following situation, which we describe informally. Start with any configuration for which $A\sm \W$ is nonempty, and add a tiny ball $B_r(x)$ to $\W$, with $x$ in the interior of $A\sm \W$ and $r$ much smaller than the distance from $x$ to $\partial(A\sm \W)$. A trivial competitor for the new $A$ is the old insulator with a small annulus deleted around $B_r (x)$; this increases the value of $F$ by an amount on the order of $r^{n-1}$. A refinement of this results in a competitor where $A$ has two components which \emph{share boundary}. If the optimal $A$ was to be connected, however, it appears unavoidable that the value of $F$ would be increased by an amount on the order of the capacity of $B_r(x)$, or $r^{n-2}$, which in the limit of $r$ small is much larger. In particular, this suggests that having two local components of $A$ sharing boundary is an unavoidable feature of some minimizers.

The first goal of this paper is to show that, for a given set $\W$, there exists a minimal set $A$. Our approach is based on the basic observation, evident from \eqref{eq:model3}, that if we know, a priori, that the minimal function $u$  is bounded from below by a strictly positive number, then the perimeter of $A$ is controlled by $F$. Such an a priori bound from below on $u$ is not trivial, and requires some use of the global structure and the volume penalization. Nevertheless, we establish this bound in Section 3.

As the example above suggests, sets of finite perimeter are not a suitable relaxed setting for this problem, as they will not properly count the perimeter of those pieces of the boundary of $A$ for which $A$ has Lebesgue density $1$. Instead, we use special functions of bounded variation, which have the same useful compactness and semicontinuity properties but offer extra flexibility. Effectively, we view both the set $A$ and the function $u$ as a function which is allowed to have codimension-$1$ jumps. These special functions of bounded variation were introduced specifically to address ``free discontinuity'' problems, the most famous of which is the minimization of the Mumford-Shah functional introduced in \cite{MS}. It turns out that our situation is quite closely related to the Mumford-Shah problem. Apart from obvious differences in the global structure, the major local difference is that the Mumford-Shah functional counts the perimeter without multiplicity, while we count the (weighted) perimeter with multiplicity (either $1$, when there is one component of $A$ bordering a point of $\partial A$, or $2$, if there are two).

More precisely, we identify the pairs $(A,u)$ above with functions $u1_A$, which are in $SBV$ at least when $\cH^{n-1}(A)<\8$. Then the energy in \eqref{eq:model3} is expressed as
\begin{equation}\label{eq:model1}
 F(u) = \int_{\R^n} |\n u|^2 d\cL^n + h \int_{S_u} \uu^2 +\ud^2 d\cH^{n-1} + C_0\cL^n(\{u>0\}).
\end{equation}
Here $\uu,\ud$ are the approximate upper and lower limits of $u$ (they are defined in the next section), $S_u$ is the singular set where $\ud<\uu$, and $\n u$ is the absolutely continuous part of the derivative of $u$. This version of the functional is defined for bounded $SBV$ functions, and when evaluated on functions of the form $u1_A$ for a smooth domain $A$ coincides with \eqref{eq:model3}. When $A$ has regular pieces of boundary with $A$ on either side, this version of $F$ will add the contribution of the trace of $u$ from each side to the surface term. One of the main results of the paper is then the following theorem, whose proof is based on the a priori estimate described above:

\begin{theorem} Let $\W$ be a bounded open set. Then there is a function $u\in SBV(\R^n)$, with $u(x)= 1$ for $\cL^n$-a.e. $x\in \W$,  such that
\[
 F(u)\leq F(v)
\]
for all $v$ with the same assumptions. Moreover, $u(x) \in \{0\}\cup [\d,1]$ for $\cL^n$-a.e. $x\in \R^n$, where $\d=\d(\W)>0$.
\end{theorem}

Using the argument of \cite{DCL} for the Mumford-Shah functional, we next show that the singular set $S_u$ of the relaxed minimizer satisfies the density estimates
\[
 cr^{n-1}\leq \cH^{n-1}(S_u \cap B_r(x)) \leq Cr^{n-1},
\]
where $x\in \bar{S}_u$ and $r<d(x,\W)$. The argument for this estimate demands only superficial changes from the Mumford-Shah case, and we include it mainly for completeness. As a consequence, we may extract from the relaxed minimizer an open set $A$ and a harmonic function $u$ which are minimizers in the ``classical'' sense.

In a companion paper \cite{Kriv}, the second author shows that at every point $x\in \bar{S_u}$ at which $S_u \cap B_r(x)$ is trapped between two parallel hyperplanes within $\e r$ of each other, for a sufficiently small $\e$, $\bar{S}_u \cap B_{r/2}$ is actually equal to the union of the graphs of two $C^{1,\a}$ functions. The remainder of this paper is devoted to finding conditions under which this hypothesis is satisfied, using blow-up techniques. The blow-ups are homogeneous of degree $\frac{1}{2}$, as this is the natural scaling of the energy; this type of blow-up was introduced by Bonnet \cite{Bon} for the Mumford-Shah functional in the plane, which has the same scaling. In our case, the blow-up limits are local minimizers of a certain limiting functional, and this global problem has rather different structure from the Mumford-Shah case.

We show that at points on the boundary of at least two local connected components of $\{u>0\}$, or at points of vanishing density of $\{u=0\}$, the set $\bar{S}_u$ is locally flat (meaning the theorem in \cite{Kriv} applies), except for a set of codimension $8$.  In the case of $n=2$, we also show that points on the boundary of a component of $\{u=0\}$ satisfy the flatness condition. Whether there are any non-flat points for planar minimizers which are accumulation points of connected components of $\{u=0\}$ is left as an open question.

There are many other questions related to this kind of free boundary problem which remain open as well. For example, it would be reasonable to expect that if $\W$ is convex, then $\{u>0\}$ is also convex, or at least star-shaped. If it was shown to be convex, it would then follow from Lemma \ref{lem:EL} that at each point $x\in \p \{u>0\}$, the blow-ups of $\{u>0\}$ are half-spaces, and so $\p \{u>0\}$ is a smooth hypersurface. It is less clear whether in dimension higher than two the property of $\p \{u>0\}$ locally being a Lipschitz graph ensures flatness, as there are cones and homogeneous functions which are stationary for the blow-up problem (in the sense of satisfying the conclusion of Lemma \ref{lem:EL} at all points) which have singularities. Finally, there are other problems which may have similar free boundary conditions: for example, if in our motivating example the temperature in $\W$ was not constant but rather some function $\phi>0$, similar considerations would lead to the problem of minimizing
\[
 \int_{\p \W} -u_{\nu_\W}d\cH^{n-1} + \cL^n(A)
\]
where $u=\phi$ on $\p \W$, $u_{\nu_A} + h u=0$ on $\p A$, and $u$ is harmonic on $A\sm \W$; see \cite{Aguilera1987} for the case of the Dirichlet condition. Another example, considered recently in \cite{Bucur2015}, features $p$-Laplace type energies and nonlinear surface terms.

The structure of the paper is as follows: Section 2 outlines the notation and reviews relevant facts from geometric measure theory. Section 3 gives a proof of the a priori lower bound on $u$, and Section 4 uses this estimate to prove the existence of minimizers. Section 5 presents the density estimate on $S_u$ and some corollaries. Section 6 gives the details for a blow-up procedure that is used in the rest of the paper. Then Section 7 collects some facts about solutions to a global minimization problem, obtained in the blow-up limit. Section 8 uses this and some results from minimal surface theory to give flatness criteria, while Section 9 derives some additional flatness criteria in the plane.

After completing this work, we learned of the recent paper \cite{Bucur2014}, which deals with the same problem. The authors there also use a relaxation to SBV, prove a lower bound on $u$ like ours from Section 3, and construct minimizers. They also present a completely different argument from the one in Section 5 for the lower density estimate on the jump set $J_u$; their method has the advantage of giving an explicit constant in this estimate. The topic of regularity and the blow-up methods in our remaining sections are not covered in \cite{Bucur2014}. We also learned of the related work \cite{Bucur2015} mentioned above and the very recent preprint \cite{2016arXiv160102146B}, which deals with some global issues for similar problems.

\section{Notation and Tools}

In this section we explain our notation and discuss the relevant properties
of the space $SBV$ of special functions of bounded variation. Most of the results here can be found in the book of Ambrosio, Fusco, and Pallara \cite{AFP}.

\subsection{Functions of Bounded Variation}

The space $BV(\mathbb{R}^{n})$ contains functions $u$ in $L^1 (\R^n)$ with distributional gradients representable by finite Borel regular measures (which we write as $Du$), and is equipped with the usual norm:
\[
 \|u\|_{BV(R^n)}=\|u\|_{L^1(R^n)}+|Du|(\R^n).
\]
For a function of bounded variation, we define the \emph{measure-theoretic upper and lower limits} by
\[
 \uu(x)=\inf\3 t \mid \limsup_{r\searrow 0}\fint_{B_{r}(x)}(u-t)_{+}d\cL^n=0\4
\]
and
\[
 \ud(x)=\sup\3 t \mid \limsup_{r\searrow 0}\fint_{B_{r}(x)}(t-u)_{+}d\cL^n=0\4.\]
 Points for which $\uu(x)=\ud(x)$ are referred to as \emph{points of approximate continuity}, while $S_u=\{x\mid \uu(x)>\ud(x)\}$ is the singular set. Notice that points of approximate continuity are precisely those $x$ for which the blow-ups $v_{x,r}(y)=u(x+ry)$ converge to constant functions in $L^1_{loc}$ as $r\searrow 0$. Sometimes we will prefer to work with $K_u=\bar{S}_u$, the topological closure.
 
 The set $J_u\ss S_u$ of \emph{jump points} for a function of bounded variation consists of those points for which the blow-ups $v_{x,r}(y)=u(x+ry)$ converge to a function constant on each side of a hyperplane; i.e. there exists a unit vector $\nu_x$ such that
 \[
  v_{x,r}\xrightarrow{L^1_{loc}} 1_{\{y\mid\langle \nu_x,y\rangle < 0\}}\ud(x)+1_{ \{y\mid\langle \nu_x,y\rangle > 0\}}\uu(x)
 \]
as $r\searrow 0$. For the properties of the sets $J_u$ and $S_u$ we refer to \cite[3.6-3.7]{AFP}; the most consequential fact about them for us is the Federer-Vol'pert theorem:

\begin{proposition}\label{prop:FedVol} Let $u\in BV(\R^n)$. Then $S_u$ is countably $\cH^{n-1}$-rectifiable, $\cH^{n-1}(S_u \sm J_u)=0$, and
\[
 Du \mres S_u = (\uu-\ud)\nu_x \cH^{n-1}\mres J_u.
\]
\end{proposition}

Using this theorem and the Lebesgue decomposition, we arrive at a representation $Du=\n u \cL^n+ (\uu-\ud)\nu_x \cH^{n-1}\mres J_u + D^{c}u$, where $\n u$ is the density of the part of $Du$ which is absolutely continuous with respect to Lebesgue measure. The third term represents the remaining singular part of $Du$, and $u$ is said to be a \emph{special function of bounded variation} if that term vanishes. We set $SBV(\R^n)$ to be the linear subspace of $BV$ containing such functions.

The essential properties of $SBV$ are summarized below. The first proposition is a compactness and closure property of $SBV$, which is a special case of \cite[Theorem 4.7, 4.8]{AFP} (except the very last statement, which instead may be deduced from Theorem 5.22; the reason that the integrand is jointly convex is explained in Example 5.23b, while the extra assumption on the integrand being bounded from below is only used to show that $\cH^{n-1}(J_{u_i})$ is uniformly bounded, which we assume separately).

\begin{proposition}\label{prop:SBVcpt}
 Let $\{u_i\}_{i\in \N}$ be a sequence of functions in $SBV(B_R)\cap L^\8 (B_R)$. Then the following hold:
 \begin{itemize}
  \item[(i)] If
  \[ 
   \sup_i \1 \int_{B_R}|\n u_i|^2 d\cL^n+\|u_i\|_{L^\8 (B_R)}+\cH^{n-1}(J_{u_i})\2<\infty,
  \]
   then there is a subsequence $u_{i_k}\rightarrow u$ in $L^1(B_R)$, and $u\in SBV$.
  \item[(ii)] Furthermore, along the subsequence above we have
  \[
   \int_{B_R}|\n u|^2 d\cL^n\leq \liminf_k \int_{B_R}|\n u_{i_k}|^2 d\cL^n,
  \]
  \[
   \cH^{n-1}(J_{u})\leq \liminf_k \cH^{n-1}(J_{u_{i_k}}),
  \]
  and
  \[
   \int_{J_u}\uu^2+\ud^2 d\cH^{n-1}\leq \liminf_k \int_{J_{u_{i_k}}}\uu^2_{i_k}+\ud^2_{i_k} d\cH^{n-1}.
  \]
 \end{itemize}
\end{proposition}

\subsection{Sets of Finite Perimeter, Hausdorff Convergence, and Uniform Rectifiability}

A Borel set $E\ss \R^n$ is said to have finite perimeter if $1_E$ lies in $BV$. Given such a set $E$, we define the Borel measure
\[
 P(E;A)=|D1_E|(A).
\]

For an arbitrary set $E$, set $E^{(0)}$ to be the set of points with Lebesgue density $0$ for $E$, and $E^{(1)}$ the set of points of Lebesgue density $1$. The \emph{essential boundary} is defined by $\p^e E = \R^n \sm (E^{(0)} \cup E^{(1)})$, and if $E$ has finite perimeter, this agrees with $S_{1_E}$. The \emph{reduced boundary} $\partial^{*}E$ of a set of finite perimeter is the set of points whose blow-ups converge to half-spaces, or $J_{1_E}$. From Proposition \ref{prop:FedVol}, we have that $\cH^{n-1}(\p^e E\sm \p^* E)=0$.

Given two closed sets $A,B\ss U$ and an open set $V \cc U$, define the distance
\[
 d(A,B;V)= \sup\{d(x,A)|x\in V\cap B\} + \sup\{d(x,B)|x\in V\cap A\}.
\]
We say that a sequence of closed sets $A_k \rightarrow A$ \emph{locally on $U$ in Hausdorff topology} if $d(A_k,A;V)\rightarrow 0$ for every $V\cc U$. The useful fact about this definition is that given a family of closed sets in $U$, we may always find a sequence which converges to some closed set in this sense. If we have a sequence of closed sets $A_k$ such that $\cH^{n-1}\mres A_k \rightharpoonup \mu$ weakly-* as measures on $U$, and moreover
\begin{equation}\label{eq:uniformden}
 cr^{n-1}\leq \cH^{n-1}(B_r(x)\cap A_k) \leq Cr^{n-1}
\end{equation}
for each $B_r(x)\ss U$ with constants uniform in $x,r,$ and $k$, then $A_k \rightarrow \supp \m$ locally on $U$ in Hausdorff topology, and the limit will satisfy the same density estimate.

There are many equivalent definitions of uniform rectifiability (see \cite{DS}), but we select one which will help explain the one property we require of them. We say that a measurable function $\w:\R^{n-1}\rightarrow [0,\infty]$ is an $A_1$ weight, with constant $A$, if for every ball $B$ we have
\[
 \fint_{B}\w d\cL^{n-1} \leq A\w(x)\qquad \text{ for a.e. }x\in B.
\]
An important property of such weights, known as the reverse H\"older inequality, is that there is an exponent $p=p(A)>1$ such that
\[
 \1\fint_B \w^p d\cL^{n-1}\2^{\frac{1}{p}}\leq C(A) \fint_B \w d\cL^{n-1}.
\]
We say that a closed set $E\ss \R^n$ satisfying the density estimates \eqref{eq:uniformden} for $r\leq 1$ is \emph{uniformly rectifiable} with constant $A$ if there is a continuous map $z:\R^{n-1}\rightarrow \R^n$ and an $A_1$ weight $\w$ with constant $A$ such that the distributional derivative of $z$ is absolutely continuous,
\[
 |\n z|\leq \w^{\frac{1}{n-1}},
\]
Lebesgue-a.e., and
\[
 \int_{z^{-1}(B_r(x))}\w d\cL^{n-1} \leq A r^{n-1},
\]
such that $E\ss z(\R^n)$. Notice that any uniformly rectifiable set is rectifiable.
\begin{proposition}\label{prop:rectlimits}Let $A_k\ss \R^n$ be a sequence of closed sets converging locally in Hausdorff topology to a set $A$, and satisfying the density estimates \eqref{eq:uniformden}. Assume, moreover, that the $A_k$ are uniformly rectifiable with constant $C$. Then so is $A$.
\end{proposition}
\begin{proof}[Sketch of proof]
 Let $z_k,\w_k$ be the associated parametrization to $A_k$. It is straightforward to check that each $z_k(\R^{n-1})$ is closed and satisfies density estimates \eqref{eq:uniformden} with constants depending only on $C$, so we may assume, passing to a subsequence, that they converge to a set $Q$ in Hausdorff topology. Clearly $A\ss Q$. From the reverse H\"older inequality, we have a uniform $W^{1,(n-1)p}_{loc}$ bound on $z_k$, so we have that $z_k \rightarrow z$ locally uniformly, and $z(\R^{n-1})=Q$. Moreover, the weights $\w_k$ are locally uniformly integrable, and so $\w_k \rightharpoonup \w$ weakly in $L^1_{loc}$, $|\n z|\leq \w^{\frac{1}{n-1}}$ a.e., and $\w$ is an $A_1$ weight. Finally,
\[
 \int_{z^{-1}(B_r(x))}\w d\cL^{n-1} \leq C r^{n-1}
\]
can be checked using the uniform convergence of $z_k$ to $z$, which implies that for $k$ large enough, $|z_k - z|\leq \e$ on $B_{\e^{-1}}$. Then $z_k(z^{-1}(B_r(x))\cap B_{\e^{-1}}) \ss B_{r+\e}(x)$, and hence
\begin{align*}
 \int_{z^{-1}(B_r(x))\cap B_{\e^{-1}}}\w d\cL^{n-1} &= \lim \int_{z^{-1}(B_r(x))\cap B_{\e^{-1}}}\w_k d\cL^{n-1}\\
 &\leq  \liminf  \int_{z_k^{-1}(B_{r+\e}(x))}\w_k d\cL^{n-1}\\
 &\leq C (r+\e)^{n-1}.
\end{align*}
Now send $\e\searrow 0$.
\end{proof}

\subsection{Definitions of Energy Minimizers}

We now fix an open set $\W \ss \R^n$. This set should be thought of as the ``boundary data'' for the problem, and examples we have in mind are where $\W$ is the interior or the exterior of a smooth compact hypersurface. To avoid technicalities, we will assume that $\Omega$ is bounded, but the the reader may check that nothing changes under weaker assumptions (such as $\W \ss \W'$, $\W'$ has finite perimeter, and $\cL^n(\W' \sm \W)<\8$).

The strong formulation of the problem is as follows: consider pairs $(A,u)$ of open sets with smooth boundaries $A$ and functions $u$ with $u-1\in H^1(A)$, where $\W \subset A$ and $u=1$ on $\W$. The task is to minimize
\[
 F(A,u)=\int_{A \sm \W} |\n u|^2 d\cL^n+\int_{\partial A}u^2 d\cH^{n-1}+\cL^n(A \sm \W)
\]
over all such pairs. It is obvious (using $\max\{0,\min\{1,u\}\}$ as a competitor) that it suffices to minimize over functions $u$ taking values in $[0,1]$.

It is unreasonable to actually attempt such a minimization directly, so we relax the problem to the space $SBV$. To do so, consider the extension of $u$ by $0$ to the entire $\R^n$. This is a function in $SBV$, and $S_u$ is the portion of the boundary of $A$ where $u\neq 0$. We may therefore attempt to minimize
\[
 F(u)=\int_{\R^n}|\n u|^2d\cL^n + \int_{J_u}\uu^2+\ud^2 d\cH^{n-1}+\cL^n(\{u>0\}\sm \W)
\]
over all functions in $SBV(\R^n)$ with $0\leq u\leq 1$ and $u=1$ on $\W$.

Any such function minimizing $F$ will be referred to as a \emph{minimizer}. Any such function with the property that $F(u)\leq F(u1_{A})$ for any set of finite perimeter $A$ containing $\W$ will be called an \emph{inward minimizer}. This is a much weaker property.

Finally, a function $u$ is said to be a \emph{local minimizer on a ball} $B_r(x)\ss \R^n \sm \W$ if $F(u)\leq F(v)$ for all $v\in SBV(\R^n)$ with $u-v$ supported on $B_r(x)$. Local minimality does not depend on $\W$, in the sense that a minimizer for any $\W$ will be a local minimizer on $B_r(x)$, provided the ball is outside of $\W$. 

\section{The Key Estimate}

The main difficulty in obtaining minimizers by the direct method is
an estimate on the $n-1$ dimensional measure of the jump set. In
this section we obtain this bound, not only as an a priori estimate
on minimizers, but also for inward minimizers
(which greatly simplifies the following section).

The following lemma is technical, showing that the coarea formula applies in a low-regularity situation.

\begin{lemma}\label{lem:coarea} Let $u\in BV(\R^n)$ be a function with $0\leq u \leq 1$. Then we have
\[
 \int_{0}^{1}P(\{u>s\};\R^n \sm S_u)ds=|Du|(\R^n \sm S_u).
\]
\end{lemma}

\begin{proof}
 From the Fleming-Rishel coarea formula for $BV$ functions \cite{FR}, we have
 \[
   \int_{0}^{1}P(\{u>s\})ds=|Du|(\R^n),
 \]
so it suffices to show
\[
  \int_{0}^{1}P(\{u>s\};S_u)ds=|Du|(S_u).
\]
Also note that it follows that the level sets $\{u>s\}$ have finite perimeter for $\cL^1$ a.e. $s$. Using the Federer-Vol'pert theorem, the right-hand side is given by
\[
 |Du|(S_u)=\int_{J_u}|\uu-\ud|d\cH^{n-1}.
\]
On the other hand, from the structure of sets of finite perimeter, for almost every $s$ we have
\[
 P(\{u>s\};S_u)=\cH^{n-1}(\partial^e \{u>s\}\cap J_u)=\cH^{n-1}(\partial^* \{u>s\}\cap J_u).
\]
Now notice that we have (from looking at the blow-ups)
\[
J_u \cap \{\ud< s<\uu\} \subset \partial^e \{u>s\}\cap J_u, \qquad \partial^* \{u>s\}\cap J_u \subset J_u \cap \{\ud\leq s \leq \uu\},
\]
which implies
\[
 \int_{J_u}1_{(\ud,\uu)}(s)d\cH^{n-1}\leq P(\{u>s\};S_u)\leq \int_{J_u}1_{[\ud,\uu]}(s)d\cH^{n-1}.
\]
We now integrate in $s$ and use Fubini's theorem to obtain
\[
 \int_{0}^1 \int_{J_u}1_{(\ud,\uu)}(s)d\cH^{n-1}ds=\int_{J_u}|\uu-\ud|d\cH^{n-1}=\int_9^1 \int_{J_u}1_{[\ud,\uu]}(s)d\cH^{n-1}ds,
\]
which gives the conclusion.
\end{proof}

Fix a ball $B_R(x)$ containing $\W$, and note that $u=1_{B_{R}(x)}$ is a valid competitor function in the minimization problem for $F$. We set $\bar{F} = F(1_{B_R(x)})$, and generally restrict our attention to admissible competitor functions $v$ with $F(v)\leq \bar{F}$ (minimizing $F$ over all such functions is equivalent to minimizing over all competitors).

\begin{theorem}\label{thm:lowerbd}
 Let $u$ be an inward minimizer with $F(u)\leq 2\bar{F}$. Then there is a value $\delta=\delta(\W)>0$ such that $\cL^n(\{0<u< \delta\})=0$.
\end{theorem}

\begin{proof}
 We apply Lemma \ref{lem:coarea} to the $BV$ function $\frac{1}{2}\min\{u,t\}^2$, obtaining
 \[
  \int_{\{u\leq t\}\sm S_u}u|\n u|d\cL^n = \int_{0}^t sP(\{u>s\}; \R^n \sm S_u)ds.
 \]
The left-hand side can be estimated using the Cauchy-Schwarz inequality to give
\[
 \int_{\{u\leq t\}\sm S_u}u|\n u|d\cL^n\leq \1\int u^2 d\cL^n\2^{1/2}\1\int |\n u|^2 d\cL^n\2^{1/2}\leq F(u)\leq 2\bar{F},
\]
so in particular both sides are finite and bounded in terms of $\W$ only.
Set
\[
 f(t)=\int_{0}^t sP(\{u>s\}; \R^n \sm S_u)ds.
\]
This is an absolutely continuous function, and we will show it satisfies a differential inequality. To that end, for $t$ with $\{u>t\}$ having finite perimeter, consider the competitor $u1_{\{u>t\}}$. This gives, from inward minimality and Proposition \ref{prop:FedVol}, that
\begin{align*}
 0&\leq F(u1_{\{u>t\}})-F(u)\\
&=\int_{\partial^*\{u>t\}\sm S_u}\uu^2 d\cH^{n-1}-\int_{\{u\leq t\}\sm S_u}|\n u|^2 d\cL^n - \cL^n(\{0<u\leq t\})\\
&\qquad -  \int_{J_u \cap \{u > t\}^{(0)}}\uu^2+\ud^2 d\cH^{n-1}  -   \int_{J_u \cap \p^* \{u > t\}}\ud^2 d\cH^{n-1}.
\end{align*}
Regrouping terms,
\begin{align*}
 \int_{J_u \cap \{u > t\}^{(0)}}\uu^2+\ud^2 d\cH^{n-1}&  +   \int_{J_u \cap \p^* \{u > t\}}\ud^2 d\cH^{n-1}
+\int_{\{u\leq t\}\sm S_u}|\n u|^2 d\cL^n + \cL^n(\{0<u\leq t\})\\
&\leq t^2 P(\{u>t\};\R^n \sm S_u) = tf'(t).
\end{align*}

We are now in a position to estimate $f$, using H\"older inequality and the $BV$ Sobolev inequality applied to the function $u^2 1_{u\leq t}$:
\begin{align*}
 f(t)&\leq \cL^n(\{u\leq t\})^{\frac{1}{2n}} \1\int_{\{u\leq t\}} u^{2^*} d\cL^n\2^{\frac{1}{2^*}}\1\int_{\{u\leq t\}} |\n u|^2 d\cL^n\2^{\frac{1}{2}}\\
 &\leq C(n)(tf'(t))^{\frac{1}{2}+\frac{1}{2n}}\1\int ud|D(u1_{u\leq t})|\2^{1/2}.
\end{align*}
Here $2^*=\frac{2n}{n-1}$. The absolutely continuous portion of the integral in the last factor is estimated by Cauchy-Schwarz, while the singular portion we compute directly:
\begin{align*}
 \int ud|D(u1_{u\leq t})|&\leq \sqrt{\cL^n(\{0<u\leq t\})\int_{\{u\leq t\}\sm S_u}|\n u|^2 d\cL^n} + \int_{J_u \cap \{u > t\}^{(0)}}\uu^2 + \ud^2 d\cH^{n-1}\\
& \qquad+\int_{J_u \cap \p^* \{u > t\}}\ud^2 d\cH^{n-1}+t^2 P(\{u>t\};\R^n \sm S_u)\\
&\leq 4tf'(t),
\end{align*}
as each of the terms appears in our previous estimate. We therefore arrive at
\[
 f(t)\leq C(n)(tf'(t))^{1+\frac{1}{2n}}.
\]
For any $t>t_0$ with $f(t_0)>0$, we may rewrite this as
\[
 \frac{f'(t)}{f(t)^{\frac{2n}{2n+1}}}\geq \frac{c(n)}{t}
\]
and then integrate from $t_0$ to $1$. This gives
\[
 f(1)^{\frac{1}{1+2n}}-f(t_0)^{\frac{1}{1+2n}}\geq c(n)\log\frac{1}{t_0},
\]
or
\[
 f(t_0)^{\frac{1}{1+2n}}\leq f(1)^{\frac{1}{1+2n}}+c(n)\log(t_0)\leq (2\bar{F})^{\frac{1}{1+2n}}+c(n)\log(t_0).
\]
Now if $t_0<\delta=\exp (-\bar{F}^{\frac{1}{1+2n}}/c(n))$, we obtain $f(t_0)<0$, which is a contradiction. It must therefore be the case that $f(t)=0$ for $t\leq \delta$, from which it follows that $\cL^n(\{0<u\leq \delta\})=0$.
\end{proof}

\begin{corollary}\label{cor:radius}
There is a  value $0<\d_0=\d_0(\W)\leq\d$ such that any inward minimizer $u$ with $F(u)\leq 2\bar{F}$ is supported on $B_{1/\d_0}$. Moreover, for any $B_r(x)\ss \R^n \sm \W$, we have
\[
 \cH^{n-1}(J_u\cap B_r(x))\leq C(n,\W)r^{n-1},
\]
while if $x\in K_u$, we also have
\[
 \cL^n(B_r(x)\cap \{u>0\})\geq c(n,\W)r^n.
\]
\end{corollary}

\begin{proof}
 The main observation is that for an inward minimizer, Theorem \ref{thm:lowerbd} implies that
 \[
  \int_{J_u\cap B_r(x)}\uu^2+\ud^2 d\cH^{n-1}\geq \d^2 \cH^{n-1}(J_u\cap B_r(x)),
 \]
 while by using $u1_{\R^n \sm B_r(x)}$ as a competitor, the left quantity is controlled by $P(B_r(x);\{u>0\}^{(1)})\leq C(n)r^{n-1}$.

 The other points now follow in a standard way. We have that, for almost every $r$,
 \[
  \cL^n(B_r(x)\cap \{u>0\}^{(1)})\leq C(n)P(B_r(x)\cap \{u>0\}^{(1)})^{\frac{n}{n-1}}\leq C(n,\W)P(B_r(x);\{u>0\}^{(1)})^{\frac{n}{n-1}}.
 \]
The assumption that $x\in K_u$ guarantees that the left-hand side is nonzero, and so by recognizing that  the left side is absolutely continuous as a function of $r$ and $P(B_r(x);\{u>0\}^{(1)})=\partial_r \cL^n(B_r(x)\cap \{u>0\}^{(1)})$, we may integrate this differential inequality to reach the conclusion.

Finally, if there is a point $x\in K_u$ at least a distance of $1/\d_0$ from $\W$, we obtain that
\[
 c(n,\W)\d_0^{-n} \leq \cL^n(\{u>0\}\sm \W)\leq F(u)\leq 2\bar{F},
\]
which is a contradiction for $\d_0$ too small.
\end{proof}

\section{Existence of Minimizers}

In this section we use the key estimate to apply the direct method
to our problem. Let 
\[
H_a=\{u\in SBV \mid u(x)=1 \forall x\in \W, u(x)\in \{0\}\cup [a,1]\text{ for } \cL^n \text{ a.e. } x, \supp u \ss B_{1/a}\}.                     
\]
With this notation, we are searching for a minimizer of $F$ amongst all of $H_0$. The following lemma says that it is equivalent to minimize over $H_a$ for small positive $a$ instead.

\begin{lemma}\label{lem:wkmin}
 A function $u\in H_0$ is a minimizer if and only if $F(u)\leq F(v)$ for all $v\in H_{\delta_0}$, where $\delta_0$ is the constant in Theorem \ref{thm:lowerbd}.
\end{lemma}

\begin{proof}
 The only if implication is trivial.
 Consider any function $v\in H_0$ with $F(v)\leq \bar{F}$, and a truncation $v_a = v1_{\{v\leq a\}\cap B_{1/a}}$ where $a$ is a small positive number. Estimating crudely,
 \[
  |F(v_a)-F(v)|\leq a^2 (P(\{v>a\})+\cH^{n-1}(\partial B_{1/a} \cap {\vu \geq a})).
 \]
On the other hand, integrating in polar coordinates, using Chebyshev's inequality and the Fleming-Rishel coarea formula,
\[
\int_{1/t}^\infty \frac{1}{r}\cH^{n-1}(\{v\geq 1/r\})dr + \int_0^t P(\{v>a\})da=\|v\|_{L^1}+ |Dv|(\{u\leq t\})\leq C(v),
\]
the constant being the $BV$ norm. Changing variables in the first integral, we see that
\[
 \int_0^t \frac{1}{a}\cH^{n-1}(\{v\geq a\})+ P(\{v>a\})da\leq C(v).
\]
In particular, for any $t\in (0,1)$ there is some $a<t$ with
\[
\cH^{n-1}(\{v\geq a\})+ P(\{v>a\})\leq C(v)/t\leq C(v)/a.
\]
We may therefore find a sequence $a_k\searrow 0$ with
\[
 F(v_{a_k})\leq F(v)+\frac{1}{k}.
\]

Now consider the following auxiliary minimization problem for sets of finite perimeter $A$ containing $\W$ and contained in $\{v_{a_k}>0\}$: minimize
\[
 G_k(A)=F(v_{a_k} 1_A).
\]
Let $\{A_j\}_{j}$ be a minimizing sequence. Clearly these sets have equibounded diameter, being contained in $B_{1/a_k}$. Now observe that
\[
 G_k(A)=\cL^n (A\sm \W)+\int_A |\n v|^2 d\cL^n + \int_{J_{1_A v_{a_k}}}(\overline{v_{a_k}1_A})^2+(\underline{v_{a_k}1_A})^2 d\cH^{n-1}, 
\]
and the last term bounds $a_k^2 \cH^{n-1}(J_{1_A v_{a_k}})\geq  a_k^2 P(A)$. We therefore infer from the compactness of sets of finite perimeter that along a subsequence, $A_j\rightarrow A^{(k)}$ in $L^1$ and that the functions $v_{a_k}1_{A_j}\rightarrow w_k=v_{a_k}1_{A^{(k)}}$. From Proposition \ref{prop:SBVcpt}, we then have that $F$ is lower semicontinuous along the subsequence, and
\[
 F(w_k)\leq \inf_{\W\ss A\ss \{v_{a_k}>0\}} F(v_{a_k}1_A)\leq F(v_{a_k})\leq F(v)+\frac{1}{k}.
\]

The above inequality implies that the function $w_k$ is an inward minimizer. Applying Theorem \ref{thm:lowerbd} and Corollary \ref{cor:radius}, we get that $w_k\in H_{\d_0},$ and hence a valid competitor for $u$. This gives
\[
 F(u)\leq F(w_k)\leq F(v)+\frac{1}{k};
\]
sending $k\rightarrow \8$ completes the proof.
\end{proof}

\begin{theorem}
 For any $\W$, there exists a minimizer $u$. Moreover, $u\in H_{\d_0}$. 
\end{theorem}

\begin{proof}
 By Lemma \ref{lem:wkmin}, it suffices to find a minimizer over $H_{\d_0}$  instead. Let $u_k\in H_{\d_0}$ be a minimizing sequence for $F$. It follows that $\uu_k(x)\geq {\d_0}$ for every $x\in J_{u_k}$, and so we have
 \[
  \d_0^2\cH^n(J_{u_k})+\int|\n u_k|^2 d\cL^n\leq F(u_k)\leq C(\W).
 \]
 It follows that the assumptions of Proposition \ref{prop:SBVcpt} are satisfied, and so $u_k\rightarrow u\in H_{\d_0}$ in $L^1$ with
 \[
  F(u)\leq \liminf_k F(u_k).
 \]
Therefore $u$ is a minimizer, and the proof is complete.
\end{proof}

\section{Density Properties}

Here we establish two density properties of minimizers: the uniform
bounds on the density of perimeter (``Ahlfors regularity''), and
the uniform positive density of each local component of the positive
phase. The first, while somewhat deeper, is well-known in this kind
of problem, and indeed parallels the very general statement
in the book of Guy David \cite{D}. The second is an elementary corollary, but
is quite specific to our situation. For convenience, we will use the notation
\[
 F(u;E)= \int_{E}|\n u|^2 d\cL^n + \int_{J_u \cap E} \uu^2 +\ud^2 d\cH^{n-1} + \cL^n (\{u>0\}\cap E)
\]

\begin{theorem}\label{thm:density}
 Let  $u\in H_{s}$ be a minimizer. Then there is a constant $0<\d_1=\d_1(\W)$ such that for any $x\in K_u$ and $B_r(x)\ss \R^n\sm \W$,
 \[
  \cH^{n-1}(J_u \cap B_r(x))\geq \d_1 r^{n-1}.
 \]
\end{theorem}

The proof of this theorem is completely analogous to the proof of the corresponding statement for $SBV$ minimizers, first shown in \cite{DCL}, of the Mumford-Shah functional, and is based on the following partial Poincar\'e inequality, which is taken from \cite[Theorem 4.14]{AFP}:

\begin{proposition}\label{prop:Poincare} Let $u\in SBV(B_r)$. Then there are numbers $-\8 < \t^- \leq m \leq \t^+ <\8$ such that the function $\tilde{u}=\max\{\min\{u,\t^+\},\t^-\}$ satisfies the Poincar\'e inequality
\[
 \|\tilde{u}-m \|_{L^{\frac{2n}{n-2}}(B_r)} \leq C \|\nabla u\|_{L^2(B_r)},
\]
and
\[
 \cL^n(\{u\neq \tilde{u}\})\leq C \1\cH^{n-1}(S_u \cap B_r)\2^{\frac{n}{n-1}}.
\]
The constants depend only on $n$.
\end{proposition}

We now follow the original argument of \cite{DCL}.

\begin{lemma}\label{lem:decay} Let $u\in H_{s}$ be a local minimizer in $B_r(x)$, and $\t$ sufficiently small. There are values $r_0,\e_0$ depending only on $n,\t$ and $s$ such that if $r<r_0$,
\[
 \cH^{n-1}(S_u \cap B_r(x))\leq \e_0 r^{n-1},
\]
and
\begin{equation}\label{eq:decayh2}
 F(u;B_r(x)) \geq r^{n-\frac{1}{2}}, 
\end{equation}
then
\[
 F(u;B_{\t r}(x)) \leq \t^{n-\frac{1}{2}} F(u;B_r(x)).
\]
\end{lemma}

\begin{proof}
 We may take $x=0$. Assume that the conclusion fails; then for any $\t>0$ there is a sequence $u_k\in H_s$ of local minimizers on $B_{r_k}$ with $r_k \searrow 0$, which satisfy
\[
 \frac{\cH^{n-1}(S_{u_k}\cap B_{r_k})}{r^{n-1}} = \e_k \searrow 0,
\]
and yet
\begin{equation}\label{eq:decayint}
  F(u_k,B_{\t r_k}) > \t^{n-\frac{1}{2}} F(u_k,B_{ r_k}).
\end{equation}
Set
\[
 E_k = r_k^{2-n}F(u_k;B_{r_k}),
\]
and define the functions
\[
 v_k (x) = \frac{u_k(r_k x)}{E_k^{1/2}}.
\]
These $v_k$ are in $SBV(B_1)$ and have
\[
 \int_{B_1}|\n v_k|^2 d\cL^n \leq 1.
\]
Let $\tilde{v}_k$ be the truncations of $v_k$ given in Proposition \ref{prop:Poincare}, and $m_k$ the corresponding medians ($\tilde{u}_k$ will refer to the scaled-back function $\tilde{u}_k(x)=\frac{E_k^{1/2}}{r_k} v_k(x/r_k)$). Then
\[
 \int_{B_1}|\tilde{v}_k-m_k|^2 d\cL^n \leq C,
\]
and
\[
 \cL^n(\{v_k \neq \tilde{v}_k\})\leq C \1\cH^{n-1}(S_{v_k})\2^{\frac{n}{n-1}}\leq C \e_k^{\frac{n}{n-1}}\rightarrow 0.
\]
It follows easily from the $SBV$ compactness theorem (Proposition \ref{prop:SBVcpt}) that, passing to a subsequence, $\tilde{v}_k-m_k$ converge in $L^2$ and a.e. to a function $v\in W^{1,2}(B_1)$, with
\[
 \int_{B_r}|\n v|^2 d\cL^n \leq \liminf_k \int_{B_r}|\n \tilde{v}_k|^2 d\cL^n
\]
for each $r\leq 1$.

Next, Consider the sequence of monotone functions
\[
 \a_k(t) = \frac{F(u_k; B_{t r_k})}{F(u_k;B_{r_k})}.
\]
Up to passing to a subsequence, we may assume that $\a_k \rightarrow \a$ $\cL^1$-a.e., where $\a$ is also monotone. Then
\[
 \int_{B_r}|\n v|^2 d\cL^n \leq \liminf \int_{B_r}|\n \tilde{v}_k|^2 d\cL^n \leq  \liminf \a_k(r)\leq 1,
\]
and by the contradiction assumption \eqref{eq:decayint} we have that $\a_k(\t)\geq \t^{n-\frac{1}{2}}$. We may take instead
\[
 \tilde{\a}_k(t) = \frac{F(\tilde{u}_k; B_{t r_k})}{F(u_k;B_{r_k})},
\]
which are monotone and bounded by
\[
 \tilde{\a}_k(t)\leq \a_k(t) + \frac{2\cH^{n-1}(J_{u_k} \cap B_{t r_k})}{F(u_k;B_{r_k})}\leq (1+\frac{2}{s^2})\a_k(t).
\]
We also assume that $\tilde{\a}_k \rightarrow \tilde{\a}$ for $\cL^1$-a.e. $t$.

Notice that, up to a subsequence,
\[
 \frac{r_k}{E_k}\1\cL^{n}(\{\tilde{v}_k\neq v_k \})  \2   \rightarrow 0
\]
as $k\rightarrow \infty$. Indeed, it follows from the estimate in Proposition \ref{prop:Poincare} that
\[
 \frac{r_k}{E_k}\1\cL^{n}(\{\tilde{v}\neq v \})  \2\leq C\frac{r_k}{E_k} \e_k^{\frac{n}{n-1}}.
\]
We are done unless $r_k/E_k \rightarrow \8$, but in that case, using $\e_k \leq s^2 F(u_k,B_{r_k}) r_k^{1-n}$, we have
\[
 \frac{r_k}{E_k} \e_k^{\frac{n}{n-1}} \leq \frac{r_k}{E_k}\1 s^2 \frac{E_k}{r_k} \2^{\frac{n}{n-1}} \rightarrow 0.
\]
It follows from Fubini's theorem that for $\cL^1$-a.e. $\r$,
\[
 \frac{r_k}{E_k}\1\cH^{n-1}(\{\tilde{v}_k\neq v_k \}\cap \partial B_{\r})\2\rightarrow 0.
\]
Let $I$ be the full-measure subset of $[0,1]$ containing those points for which both the above holds and $\a,\tilde{\a}$ are continuous.

We now show that $v$ is harmonic. Indeed, let $\phi\in W^{1,2}(B_1)$ be a function with $v-\phi$ supported on $B_\r$, where $\r<1$.  Choose a smooth cutoff function $\eta$ which is compactly supported on $B_{\r'}$ and equals $1$ on $B_\r$, and set 
\[                                                                                                                                                                                                                                                  
\phi_k = \1(m_k + \phi)\eta  + \tilde{v}_k (1-\eta)\2 1_{B_{\r'}} + v_k 1_{B_{1}\sm B_{\r'}}.                                                                                                                                                                                                                                                   \]
 We use $E_k^{1/2}\phi_k(x/r_k)$ in the minimality inequality for $u_k$ to obtain (after scaling)
\begin{align*}
 \a_k(\r')=\frac{F(u_k;B_{\r' r_k})}{F(u_k;B_{r_k})} & \leq 2\frac{r_k}{E_k}\cH^{n-1}(\{v_k \neq \tilde{v}_k\}\cap \p B_\r) + \int_{B_\r} |\n \phi|^2 d\cL^n + C[\tilde{\a}_k(\r')-\tilde{\a}_k(\r)] \\
& \qquad+ C\int_{B_{\r'}\sm B_{\r}} |\n \phi|^2 + (\tilde{v}_k- m_k -\phi)^2 |\n \eta|^2 d\cL^n + Cr_k^2 E_k^{-1}.
\end{align*}
Making sure that $\r,\r'\in I$ (note that this implies that $\lim_k \a_k(\r')=\a(\r')$) and taking the liminf in $k$, we see that the last two terms vanish: the first because $\tilde{v}_k- m_k \rightarrow v$ in $L^2$ and $v=\phi$ on $B_{\r'}\sm B_\r$, while the second from the assumption \eqref{eq:decayh2}. The first term on the right also drops, and so we are left with
\[
 \int_{B_{\r'}}|\n v|^2 d\cL^n \leq \a(\r') \leq  \int_{B_\r} |\n \phi|^2 d\cL^n + C[\tilde{\a}(\r')-\tilde{\a}(\r)] + C\int_{B_{\r'}\sm B_{\r}} |\n \phi|^2d\cL^n.
\]
Taking the limit as $\r\nearrow \r'$ gives
\[
  \int_{B_{\r'}}|\n v|^2 d\cL^n \leq \int_{B_\r'} |\n \phi|^2 d\cL^n,
\]
which implies $v$ is harmonic. Also, by plugging in $\phi=v$, we see that
\[
\a(\r')=\int_{B_{\r'}}|\n v|^2 d\cL^n. 
\] 
This equality holds on $I$, and as both sides are monotone and one of them is continuous, it also holds for every radius in $[0,1]$. Finally, $\a(1)=1$, while from \eqref{eq:decayint} and the now-established continuity of $\a$, we have $\a(\t) =\lim \a_k(\t)\geq \t^{n-\frac{1}{2}}$. For sufficiently small $\t$, this contradicts the fact that $v$ is harmonic.
\end{proof}

\begin{proof}[Proof of Theorem \ref{thm:density}]
 We prove the theorem under the restriction 
 \[
r<\min\{r_0(\t_1),r_0(\t_2),\e_{\t_2}^2 \t_1^{2n-2}\t_2^{2n-1}\},  
 \]
 with $\t_1,\t_2$ determined below; the general case follows by taking a possibly smaller constant $\d_1$ depending only on $\W$. First, let $x\in S_u$ and $B_r(x)\ss \R^n \sm \W$, and say that
\[
 \cH^{n-1}(S_u \cap B_r(x))<\e_0(\t_1) r^{n-1}.
\]
We claim that this implies that
\[
 F(u;B_{\t_1 \t_2^k r})\leq \e(\t_2)\t_2^{k(n-\frac{1}{2})} \t_1^{n-1} r^{n-1}.
\]
Indeed, for $k=0$, we have that either
\[
 F(u;B_{\t_1 r})\leq F(u;B_{r})\leq r^{n-\frac{1}{2}}\leq \e(\t_2)(\t_1 r)^{n-1},
\]
or else
\[
 F(u;B_{\t_1 r})\leq \t_1^{n-\frac{1}{2}} F(u;B_{r})\leq C(n)\t_1^{n-\frac{1}{2}} r^{n-1}\leq \e(\t_2)(\t_1 r)^{n-1},
\]
where we used the trivial estimate $F(u;B_r)\leq P(B_r)$ and then chose $C(n)^2 \t_1 = \e(\t_2)^2$. If the inequality holds for some positive $k$, then using Lemma \ref{lem:decay} again gives either
\[
 F(u;B_{\t_1 \t_2^{k+1} r}) \leq F(u;B_{\t_1 \t_2^k r}) \leq (\t_1 \t_2^k r)^{n-1/2}\leq  \e(\t_2)\t_2^{(k+1)(n-\frac{1}{2})} \t_1^{n-1} r^{n-1},
\]
or
\[
 F(u;B_{\t_1 \t_2^{k+1} r}) \leq \t_2^{n-\frac{1}{2}}F(u;B_{\t_1 \t_2^k r})\leq \e(\t_2)\t_2^{(k+1)(n-\frac{1}{2})} \t_1^{n-1} r^{n-1}
\]
from the inductive hypothesis.

It follows that $r^{1-n}F(u;B_r(x))\searrow 0$ as $r\searrow 0$, which means
\[
 \lim_{r\rightarrow 0} r^{1-n} \1\int_{B_r}|\n u|^2d\cL^n + \cH^{n-1}(S_u\cap B_r)\2 = 0.
\]
Using \cite[Theorem 7.8]{AFP}, this contradicts $x\in S_u$.

Finally, if $x\in K_u$ and
 \[
 \cH^{n-1}(S_u \cap B_{2r}(x))<\e_0(\t_1) r^{n-1},
\]
we may find $y\in S_u\cap B_r(x)$, for which 
\[
 \cH^{n-1}(S_u \cap B_r(y))<\e_0(\t_1) r^{n-1},
\]
and we are lead to a contradiction.
\end{proof}

This theorem has a number of useful consequences, some of which are discussed below and in the next section.

\begin{corollary}\label{cor:volden}
 Let $u\in H_{s}$ be a minimizer. Then $\cH^{n-1}(K_u\sm J_u)=0$. Furthermore, for any ball $B_r(x)\subset \R^n \sm \W$, $U=\{u>0\}\sm K_u $ is open and has at most $N=N(\W)$ connected components which intersect $B_{r/2}$ nontrivially, and each such component has Lebesgue measure of at least $c(\W)r^n$.
\end{corollary}

\begin{proof}
 The first conclusion is a standard fact about sets with uniform $\cH^{n-1}$ density bounds (see \cite{AFP}), while the second follows from the first by applying the argument in Corollary \ref{cor:radius} to each connected component of $B_r\sm K$ which intersects $B_{r/2}$. 
\end{proof}

\begin{remark}
 Set $A$ to be the set $\{\uu>0\}\sm K_u$. We have that $\p A \ss K_u$, as any of the boundary points of $A$ which are not in $K_u$ must have an entire small ball around them on which $u$ is harmonic, and hence strictly positive. This implies that $A$ is open. We actually have that $K_u = \p A$, as certainly $J_u \ss \p A$, and $K_u = \bar{J}_u$. From \cite[Proposition 4.4]{AFP} and the subsequent remarks, any $H^1$ function $v$ on an open set $A$ with $\cH^{n-1}(\partial A)<\infty$ and with $v$ having summable boundary trace corresponds to an $SBV$ function. It follows that our pair $(A,u)$ is a minimizer of
\[
 F(A,u) = \int_{A}|\n u|^2 d\cL^n + \int_{\p A} \uu^2 + \ud^2 d\cH^{n-1} +\cL^n(A)
\]
over all $A$ as above containing $\W$ and $u\in H^1(A)$ with $u\equiv 1$ on $\W$. 
\end{remark}

\begin{remark}\label{rem:unifrect} In his book \cite{D}, Guy David considers an extremely general notion of quasiminimizer for the Mumford-Shah functional. Indeed, it may be checked, by applying the previous remark, that any local minimizer on $B_r$ in $H_s$ is in David's class $TRLQ(B_r)$ with constant $M$ depending only on $s$ and $a$ as small as one wishes, by choosing $r_0$ small. We may then deduce from his Theorem 74.1 that $K_u$ is uniformly rectifiable. This will be used in the next section.  
\end{remark}

\section{Behavior Under Limits}

This section collects a number of auxiliary propositions relating
to how sequences of minimizers converge to other minimizers. The approach is similar to that of Bonnet, and as there are some obvious topological issues with the limiting minimization problem, we interpret it in a stronger sense than the $SBV$ framework.

Before commencing with the actual blow-up procedure, we require an improved semicontinuity property of the surface energy. The proof given here relies on \cite[Theorem 1.1]{Kriv}, which we state in the following proposition for convenience. While we note that it is possible to prove the lemma without resorting to this theorem by using the arguments of \cite[Lemma 5.4]{Kriv} along with reductions from geometric measure theory, we give the simpler proof below.

\begin{proposition}\label{prop:fimpsm}
 Let $(K,u)$ be a local $F$ minimizer on $B_r$, with $u\in \{0\}\cup [s,1]$ and $0\in K$. Then there are constants $r_0$, $\a$ and $\e$, depending only on $n$ and $s$, such that if for some $\r \leq r_0$ there exists a hyperplane $\pi$ so that
 \[
  \sup_{x\in K \cap B_\r} d(x,\pi)< \e \r,
 \]
then there are two functions $g_-,g_+: \R^{n-1}\rightarrow \R$ with $g_-\leq g_+$,
\[
 \|g_+\|_{C^{1,\a}(B_\r)} + \|g_-\|_{C^{1,\a}(B_\r)}\leq 1,
\]
and (identifying $\pi$ with $\R^{n-1}$) $K\cap B_{\r/2} = B_{\r/2} \cap \{(x',g_+(x')),(x',g_-(x'))| x'\in \R^{n-1}\}$. If $u=0$ at a point of $B_{\r/2}$ outside of the region $g_-<g_+$, then $g_-=g_+$, and is in fact a $C^\8$ function.
\end{proposition}

\begin{lemma}\label{lem:semicontinuity} Let $u_k$ be a sequence of local minimizers on $B_{2 r_k}$, with $r_k$ going to $0$, $\tilde{u}_k(\cdot):=u(r_k \cdot)\rightarrow u_0$ in $L^1$, and $K_{u_k}\rightarrow K$ locally in the Hausdorff topology on $B_2$. Assume that $K$ is a countably $\cH^{n-1}$-rectifiable set. Then
\[
 \int_{B_1 \cap K}\uu^2_0+\ud^2_0 d\cH^{n-1}\leq \liminf \int_{B_1\cap K_{\tilde{u}_k}} \overline{\tilde{u}}_k^2+\underline{\tilde{u}}_k^2 d\cH^{n-1}.
\]
\end{lemma}

\begin{proof}
 Let $J\ss K\cap B_\r$, with $\r<1$, contain those points $x\in K$ at which $K$ admits a unique tangent plane; we have that $\cH^{n-1}(K\sm J)=0$ (this is because $K$ inherits the uniform perimeter density estimate, and is by assumption countably $\cH^{n-1}$-rectifiable). Fix some $\g>0$, and find for each $x\in J$ a radius $r(x,\g)$ such that
 \[
  K\cap B_{r(x,\g)}(x) \ss \{|(x-y)\cdot \nu_x|\leq \g r(x,\g)\},
 \]
where $\nu_x$ is a unit normal to the tangent plane to $K$ at $x$. Applying Proposition \ref{prop:fimpsm} at every $x\in J$ (with $\g$ less than the $\e$ of that proposition), we find that (for $k>K(x,\g)$ large)   $B_{r(x,\g)/2}(x)\cap K_{\tilde{u}_k}$ is given by the union of the graphs of a pair of $C^{1,\a}$ functions $g_{\pm,k} :\R^{n-1}\rightarrow\R$. Moreover, $g_{\pm,k} \rightarrow g_{\pm}$ in $C^1$, $\tilde{u}_k \rightarrow u_0$ uniformly on $B_{r(x,\g)/2}(x)$, and for all $t<r(x,\g)/4$,
\[
 \int_{B_{t}(x) \cap K} \uu^2_0 + \ud^2_0 d\cH^{n-1} \leq \liminf \int_{B_{t}(x) \cap K_{\tilde{u}_k}} \overline{\tilde{u}}_k^2+\underline{\tilde{u}}_k^2 d\cH^{n-1}.
\]
 
 For each $x\in J$ we may find a sequence of shrinking radii $r_i(x)<\min\{r(x,\g)/4,\a\}$ (with $r_i(x)$ decreasing to $0$) for which the above property will hold; let $\cO_\a$ be the set of all  balls $B_{r_i(x)}(x)$. Then $\cO_\a$ is a fine cover of $J$, so by the Vitali-Besicovitch covering theorem we may find a countable disjoint subcollection $\{\bar{B}_i\}_{i}$ with radii $r_i<\a$ and $\cH^{n-1}( J \sm \cup_i \bar{B}_{i})=0$. We also have that as on $2 B_i$, $K$ is given by two graphs, $\cH^{n-1}(\p B_i\cap K)=0$. On each of the balls, we have that
\[
  \int_{K\cap B_i}\uu^2_0+\ud^2_0 d\cH^{n-1}\leq \liminf \int_{B_i\cap K_{\tilde{u}_k}}\overline{\tilde{u}}_k^2+\underline{\tilde{u}}_k^2 d\cH^{n-1}.
\]
Using Fatou's lemma and the fact that the balls are disjoint,
\begin{align*}
 \int_{K\cap B_\r}\uu^2_0+\ud^2_0 d\cH^{n-1} &\leq \sum_i \int_{K\cap B_i}\uu^2_0+\ud^2_0 d\cH^{n-1} \\
 &\leq \sum_i \liminf \int_{B_i\cap K_{\tilde{u}_k}}\overline{\tilde{u}}_k^2+\underline{\tilde{u}}_k^2 d\cH^{n-1}\\
 &\leq \liminf \sum_i \int_{B_i\cap K_{\tilde{u}_k}}\overline{\tilde{u}}_k^2+\underline{\tilde{u}}_k^2 d\cH^{n-1}\\
 &\leq  \liminf \int_{B_1\cap K_{\tilde{u}_k}} \overline{\tilde{u}}_k^2+\underline{\tilde{u}}_k^2 d\cH^{n-1},
\end{align*}
provided $\r+\a<1$. This gives the conclusion when $\r\rightarrow 1$.
\end{proof}

\begin{lemma}\label{lem:blowup}
 Let $u_k\in H_s$ be a sequence of local minimizers on $B_{r_k}(x_k)$, and $s_k$ is a sequence with $r_k/s_k\rightarrow \8$ and $s_k\rightarrow 0$. Then the following hold, along some subsequence:
 \begin{enumerate}
  \item $\frac{K_{u_k}-x_k}{s_k}\rightarrow K$ in local Hausdorff topology. The set $K$ satisfies the density estimates
  \begin{equation}\label{eq:density}
   cr^{n-1}\leq \cH^{n-1}(B_r(x) \cap K) \leq Cr^{n-1}
  \end{equation}
  for every $x\in K$ and any $r$.
  \item $u_k((\cdot-x_k)s_k) \rightarrow u_0\in SBV(\R^n)$ in $L^1_{loc}$, with $\n u_0 =0$ and $K_{u_0}\ss K$.
  \item $\R^n \sm K$ has at most $N$ connected components on which $u_0$ does not vanish, labeled $U_i$. The (constant) value of $u_0^2$ on $U_i$ is named $\m_i$. 
  \item On each component of $\R^n \sm K$, there is a sequence $c_k$ such that $\frac{u_k((\cdot-x_k)r_k)}{s_k^{1/2}}-c_k\rightarrow u$ Lebesgue a.e., as well as in $H^1_{loc}(\R^n \sm K)$. The function $u$ has locally finite $H^1$ norm.
  \item $(u,K)$ is a local minimizer on every ball $B_R$ of the functional
  \[
   F_{0}(v,L;B_R)=\int_{B_R \sm L}|\n v|^2 d \cL^n + \sum_{i=1}^{N}\mu_i \int_{\partial V_i \cap B_R}2\theta_i(x) d\cH^{n-1},
  \]
    where $L$ is a closed, countably $\cH^{n-1}$ rectifiable set with the property that if $x\in U_i\sm B_R$, $y\in U_j\sm B_R$, and $i\neq j$, then $x,y$ are not in the same component of $\R^n \sm L$. The open sets $V_i$ are the components of $\R^n \sm L$ containing $U_i \sm B_R$, and $v$ is in $H^1_{loc}(\R^n \sm L)$. The function $\theta_i$ is the Lebesgue density of $V_i$ at $x$, which is either $0,1/2,$ or $1$ at $\cH^{n-1}$-a.e. point. By local minimizer, we mean $(u,K)$ minimizes $F_0$ amongst all pairs $(v,L)$ as described for which $L=K$ and $u=v$ on $\R^n \sm B_R$.
   \item We have that
   \[
    \int_{B_R \sm K}|\n u|^2 d\cL^n = \lim s_k^{1-n} \int_{B_{Rs_k} \sm K_{u_k}}|\n u_k|^2 d\cL^n.
   \]
 \end{enumerate}
\end{lemma}

The application we have in mind is to blow-ups of a minimizer at a point. While the limiting problem may appear opaque, we will soon show that the only nontrivial situation (at least in low dimensions) is when $N=1$; $N=2$ will easily be seen to imply that $K$ is a line or pair of lines when $n=2$, while $N>2$ is not possible. The interpretation we suggest for the interesting case of $N=1$ is that of a harmonic function on a domain with variable ``holes,'' on the boundary of which a Neumann condition is satisfied. The functional minimizes the sum of the Dirichlet energy and the perimeter of the holes, with the latter counted with multiplicity (unlike for the Mumford Shah minimizers).

\begin{proof}
 Without loss of generality, we may take $x_k=0$. Now observe that, for every $R$, the scaled functions $\tilde{u}_k(x)=u_k(s_k x)$ are bounded by $1$ on $B_R$, and we also have
 \[
  \int_{B_R} |\n \tilde{u}_k|^2 d \cL^n \leq C(n)R^{n-1}s_k
 \]
and
\[
\cH^{n-1}(K_{\tilde{u}_k}\cap B_R)\leq C(n,s)R^{n-1}
\]
from using $u_k 1_{\R^n\sm B_{Rs_k}}$ as a competitor for $u_k$ and scaling. Then Proposition \ref{prop:SBVcpt} asserts that $\tilde{u_k}\rightarrow u_0\in SBV$ in $L^1_{loc}$   along a subsequence. We also obtain that
\[
 \int_{B_R} |\n u_0|^2 d \cL^n=0,
\]
while
\[
 \cH^{n-1}(K_{u_0}\cap B_R)\leq C(n,s)R^{n-1}.
\]

We now show that the sets $K_{\tilde{u}_k}$ converge to $K$ in local Hausdorff topology along a further subsequence, and that $K_0\ss K$. We may assume that $\cH^{n-1}\mres K_{\tilde{u}_k}\rightharpoonup \mu$ in the weak-$*$ sense, where $\mu$ is a Radon measure. Moreover, an elementary argument shows that $\mu$ satisfies the density estimates
\[
 cr^{n-1}\leq \mu (B_{r}(x)) \leq Cr^{n-1} 
\]
for any $x\in \supp \mu$. It follows that $K_{\tilde{u}_k}\rightarrow \supp \mu:=K$ in local Hausdorff topology. 

We now show that $J_{u_0}\subset K$. Indeed, take any $x\in J_{u_0}$ and a small ball $B_r(x)$: then we have
\[
 0<\cH^{n-1}(J_{u_0}\cap B_r(x))\leq \liminf_k \cH^{n-1}(J_{\tilde{u}_k}\cap B_r(x)),
\]
and so there is a sequence $y_k\rightarrow y$ with $y_k\in B_r(x)\cap J_{\tilde{u}_k}$, meaning $y\in K$. As this is true for each $r$, we see that $x\in \supp \mu$. This establishes properties $(1)$ and $(2)$.

We now show properties $(4)$ and $(5)$, and then return for $(3)$. For property $(4)$, fix a component $U$ of $\R^n \sm K$ on which $u_0$ is positive, and any smooth open connected $U' \cc U$. For $k$ large, $U'\ss B_R\sm K_{\tilde{u}_k}$, and in particular $\hat{u}_k=s_k^{-1/2}\tilde{u}_k$ is harmonic on $U'$. Moreover, we have our previous estimate
\[
 \int_{B_R\sm K_{\hat{u}_k}}|\n \hat{u}_k|^2 d\cL^n\leq C(n)R^{n-1},
\]
from which we may deduce that there are numbers $c_k=\fint_{U'}\hat{u}_kd\cL^n$ such that $\hat{u}_k-c_k\rightarrow u$ on $U'$ in $H^1$ norm. Now for any connected $U'\cc W\cc U$, we have the Sobolev inequality
\[
 \int_W |\hat{u}_k-c_k|^2 d\cL^n \leq C(n,U')\1\int_W |\n \hat{u}_k|^2 d\cL^n + \int_{U'} |\hat{u}_k-c_k|^2 d\cL^n \2,
\]
from which we may extract a further subsequence converging to $u$ on $W$. Repeating a countable number of times on an expanding sequence of open connected sets $W^{(k)}$ produces a (harmonic) function $u$ on $U$ which is an $H^1_{loc}$ limit of $\hat{u}_k-c_k$. We may now do this on every component of $\R^n \sm K$, and it then follows easily from monotone convergence that
\[
 \int_{B_R\sm K}|\n u|^2 d\cL^n =\lim \int_{\cup_{\text{components}} W^{(k)}}|\n u|^2\leq C(n)R^{n-1}.
\]
This establishes $(4)$.

We will require the following improved semicontinuity property of $\tilde{u}_k$:
\[
 \int_{B_R \cap K} \uu^2_0 +\ud^2_0 d\cH^{n-1}\leq \liminf \int_{B_R \cap K_{\tilde{u}_k}} \overline{\tilde{u}}^2_k +\underline{\tilde{u}}^2_k d\cH^{n-1}.
\]
The point is that the left-hand integral is over $K$, not $K_{u_0}$. To see this, recall from Remark \ref{rem:unifrect} that $K_{\tilde{u}_k}$ is uniformly rectifiable, with uniform constants. Then Proposition \ref{prop:rectlimits} implies that $K$ is uniformly rectifiable, so in particular rectifiable. Finally, the statement above follows by Lemma \ref{lem:semicontinuity}.

To show $(5)$, we begin with a competitor $(L,v)$ which coincides with $(K,u)$ outside $B_R$. Let $\eta$ be a smooth cutoff function supported on $B_{R_2}$ and equal to $1$ on $B_{R_1}$ with $R_2>R_1>R$, and $R_2-R_1$ is to be thought of as small. Let $W_\e = \{x\in B_{R_2}\sm B_{R_1}| d(x,K)<\e\}$. We will use the fact (which follow from the density estimates on $K$ and its rectifiability) that 
\[ 
\cL^n(W_\e)\leq C\e \cH^{n-1}(K\cap \{R_1-\e<|x|<R_2+\e\}), 
\]
and that, along some sequence of $\e\rightarrow 0$, that
 \[
\cH^{n-1}(\p W_\e)\leq C \cH^{n-1}(K\cap \{R_1-\e<|x|<R_2+\e\}).  
 \]
The second fact follows from the first and the coarea formula. Set
\[
 \hat{v}_k = \begin{cases}
              (1-\eta)\hat{u}_k +\eta (w_k + c_k) &  \text{ on } \R^n \sm W_\e \\
	      0 &  \text{ on } W_\e,
             \end{cases}
\]
 where $w_k=\max\{\min\{v,s_k^{-1/4}\}-s_k^{-1/4}\}$ is a truncation. Then $\hat{v}_k-c_k$ converges to $(1-\eta)u +\eta v=v$ $\cL^n$ almost everywhere, and the convergence is uniform on the set $B_{R_2}\sm (B_{R_1}\cup W_\e)$. On the other hand, if $v_k (x)=s_k^{1/2}\hat{v}_k(x/s_k)$, this is an $SBV$ competitor to $u_k$ in the local minimality inequality. After scaling, this gives:
\begin{align*}
 \int_{B_{R_2}}&|\n \hat{u}_k|^2 d\cL^n + \int_{J_{\hat{u}_k} \cap B_{R_2}}\uu_k^2 (s_k x)+\ud_k^2 (s_k x) d\cH^{n-1}(x) + s_k^{1-n} \cL^n(\{u>0\}\cap B_{R_2 s_k})\\
 &\leq  \int_{B_{R_2}}|\n \hat{v}_k|^2 d\cL^n +\int_{J_{\hat{v}_k} \cap B_{R_2}}\vu_k^2 (s_k x)+\vd_k^2 (s_k x) d\cH^{n-1}(x) + s_k^{1-n} \cL^n(\{v>0\}\cap B_{R_2 s_k}).
\end{align*}
Both of the volume terms tend to $0$ as $k\rightarrow \8$. For the energy terms, we have
\begin{align*}
 \int_{B_{R_2}}&|\n \hat{v}_k|^2-|\n \hat{u}_k|^2 d\cL^n \\
&= \int_{B_{R_2}\sm W_\e} |\n \hat{u}_k|^2 ((1-\eta)^2-1)+|\n w_k|^2 \eta^2 + |\n \eta|^2 (w_k+c_k -\hat{u}_k)^2 d\cL^n\\
&\qquad + \int_{B_{R_2}\sm W_\e}\n \hat{u}_k \cdot \n w_k (1-\eta)\eta +(\n w_k \eta +\n\hat{u}_k (1-\eta))\cdot \n \eta (w_k+c_k -\hat{u_k})d\cL^n \\
&\qquad -\int_{W_\e}|\n \hat{u}_k|^2 d\cL^n.
\end{align*}
The point is that we have $(w_k+c_k-\hat{u}_k)\rightarrow 0$ uniformly on $B_{R_2} \sm (B_{R_1}\cup W_\e)$, so taking $\liminf$ we obtain that
\begin{align*}
 \liminf\1\int_{B_{R_2}}|\n \hat{v}_k|^2-|\n \hat{u}_k|^2 d\cL^n\2 &\leq \int_{B_{R_2}} |\n v|^2 d\cL^n- \limsup \int_{B_{R_1}}|\n \hat{u}_k|^2  d\cL^n \\
&\qquad+C\liminf_{k}\int_{B_{R_2}\sm B_{R_1}}|\nabla \hat{u}_k|^2+|\n v|^2 d\cL^n.  
\end{align*}

A similar computation works for the boundary terms, for which it is convenient to go component by component $U_i\subset \R^n\sm K$. Recall that $V_i$ is the corresponding component of $\R^n \sm L$, and note that $s_k c_k^2 \rightarrow \mu_i$; this is obvious by construction.
\begin{align*}
 &\int_{J_{\hat{v}_k} \cap B_{R_2}}\vu_k^2 (s_k x)+\vd_k^2 (s_k x) d\cH^{n-1}(x)-\int_{J_{\hat{u}_k} \cap B_{R_2}}\uu_k^2 (s_k x)+\ud_k^2 (s_k x) d\cH^{n-1}(x) \\
 &\quad \leq \int_{J_{\hat{u}_k} \cap B_{R_2}\sm (B_{R_1}\cup W_\e)}2 d\cH^{n-1}(x) + \int_{J_{\hat{v}_k} \cap B_{R_2}\sm (B_{R_1}\cup W_\e)} s_k (c_k+\overline{w}_k)^2 d\cH^{n-1}(x) \\
 &+\cH^{n-1}(\p W_\e)
+ \sum_i \int_{\partial V_i \cap B_{R_1}} s_k(\overline{w}_k+c_k)^2 \theta_i^v d\cH^{n-1}- \int_{J_{\hat{u}_k}\cap B_{R_2}} \uu_k^2 (s_k x)+\ud_k^2 (s_k x) d\cH^{n-1}.
\end{align*}
The last term we recognize as the boundary integral for $\tilde{u}$, which we know is lower semicontinuous. Using in addition that $s_k |w_k^2|\rightarrow 0$,
\begin{align*}
 &\limsup\1 \int_{J_{\hat{v}_k} \cap B_{R_2}}\vu_k^2 (s_k x)+\vd_k^2 (s_k x) d\cH^{n-1}(x)-\int_{J_{\hat{u}_k} \cap B_{R_2}}\uu_k^2 (s_k x)+\ud_k^2 (s_k x) d\cH^{n-1}(x) \2\\
 &\leq 
  \sum_i \int_{\partial V_i \cap B_{R_1}} \mu_i \theta_i^v d\cH^{n-1}
  - \sum_i \int_{\partial U_i \cap B_{R_1}} \mu_i \theta_i^u d\cH^{n-1}\\
 &\qquad +C\limsup \cH^{n-1}(J_{\hat{u}_k} \cup J_{u} \cap( B_{R_2}\sm B_{R_1}))
  +C\cH^{n-1}(\p W_\e).
\end{align*}

By combining the two sets of estimates (and our estimate on the last term, from earlier), we arrive at
\begin{align*}
  \int_{B_{R_2}} &|\n v|^2 d\cL^n - \limsup \int_{B_{R_2}} |\n \hat{u}_k|^2  d\cL^n   
  + \sum_i \int_{\partial V_i \cap B_{R_1}} \mu_i \theta_i^v d\cH^{n-1}
  - \sum_i \int_{\partial U_i \cap B_{R_1}} \mu_i \theta_i^u d\cH^{n-1}\\
 &\geq -C(\a(R_2+\e)-\a(R_1-\e)),
\end{align*}
where
\[
 \a(S)=\limsup \1 \cH^{n-1}(J_{\hat{u}_k} \cap B_{S}) +\int_{B_{S}}|\nabla \hat{u}_k|^2 d\cL^n \2 + \cH^{n-1}(K \cap B_{S})+\int_{B_S}|\n v|^2 d\cL^n.
\]
Now, $\a$ is a nondecreasing function, and so is continuous at almost every value of $S$. Fix $R_2$ to be such a point of continuity, and send $\e\rightarrow 0$ and $R_1$ to $R_2$. If we take $(L,v)=(K,u)$, this gives 
\[
  \liminf_k \int_{B_R}|\n \hat{u}_k|^2 d\cL^n \geq \int_{B_R}|\n u|^2 d\cL^n \geq 
\limsup_k \int_{B_R}|\n \hat{u}_k|^2 d\cL^n\]
for almost every $R$. The middle term is continuous in $R$ and  all  three are monotone, which means that they all coincide for every $R$. Taking any other choice of $(L,v)$ establishes the minimality.

Finally, the finiteness assertion on the number of connected components in $(3)$ may be recovered by showing each has uniformly positive volume density, as in Corollary \ref{cor:volden}.
\end{proof}

In the following sections, an $F_0$ minimizer is a pair $(K,u)$ minimizing $F_0(\cdot,\cdot,B_R)$ amongst all admissible competitors (as in $(5)$) on every ball $B_R$. We will use the convention $U_0 = \{u_0=0\}\sm K$ to be the union of all of the other connected components of $\R^n \sm K$. Observe that none of the components $U_i$ for $i\geq 1$ may be bounded, for if they were, a competitor which relabels $U_i$ as a part of $U_0$ is admissible, and clearly decreases $F_0$.

\section{Properties of Global Minimizers}

In this section we describe some general properties of $F_0$ minimizers in $n$ dimensions. First, we have the following two lemmas, which compute the Euler-Lagrange equation for $F_0$ minimizers and show that $U_i$ satisfies a local one-sided perimeter minimality principle. The computation of the Euler-Lagrange equation is standard, and we only sketch the argument.

\begin{lemma}\label{lem:EL} Let $u$ be a global minimizer of $F_0$. Then for any vector field $T\in C^{1}_{c}(\R^n)$, we have
\begin{equation}\label{eq:ELc1}
 \int_{\R^n \sm K} \n u \cdot \n T \n u -2 |\n u|^2 \dvg T d\cL^2 = \sum_i \int_{\partial U_i} 2\mu_i \theta_i \dvg\nolimits^{K}T d\cH^{n-1}. 
\end{equation}
Here $\dvg^K$ is the tangential divergence. If $K\cap B_r(x)$ is given by the graph of a $C^{2}$ function $g:\R^{n-1}\rightarrow \R$, $x\in K$, and $K$ separates $B_r(x)\sm K$ into $V_1$ and $V_2$,
\begin{equation}\label{eq:ELc2}
 -(\m_{V_1}+\m_{V_2})H_{V_1}= |\n^K u|_{V_1}|^2-|\n^K u|_{V_2}|^2.
\end{equation}
If $V_i \nsubseteq U_0$, then we have in addition
\[
 -\m_{V_i} H_{V_i}\geq |\n^K u|_{V_i}|^2.
\]
Here $H_{V_i}$ is the mean curvature of $\partial_{V_i}$, $\n^K u|_{V_i}$ is the tangential derivative of the trace of $u$ from $V_i$ on $K$, and likewise for $j$.
\end{lemma}

Note that $u$ is harmonic in $\R^n\sm K$ and solves the Neumann problem near smooth points of $K$, so under the assumption that $K$ is locally $C^2$, we have that $u\in C^{1,\a}$, and the derivatives make sense classically.

\begin{proof} Let $T$ be a vector field in $T\in C^{1}_{c}(B_R)$, and set $\phi_t(x)=x+tT(x)$; for small values of $t$, this is a diffeomorphism which is the identity outside $B_R$. Then the pair $(\phi_t (K),u\circ \phi_t^{-1})$ is an admissible competitor for $(K,u)$, and we obtain the relation
\begin{equation}\label{eq:ELopt}
 F_{0}(K,u;B_R)\leq F_{0}(\phi_t (K),u\circ \phi_t^{-1};B_R).
\end{equation}
We now expand the right-hand side. For the energy term, note that $u$ is harmonic on $\R^n \sm K$, so in particular smooth. We therefore obtain
\begin{align*}
 \int_{B_R\sm \phi_t (K)}|\n (u\circ \phi_t^{-1})|^2& d\cL^n = \int_{B_R \sm K} |\n u \n\phi_t^{-1}\circ \phi_t |^2 |\det \n \phi_t|d\cL^n\\
&=\int_{B_R \sm K} |\n u(I+t\n T)^{-1}|^2(1 + t\dvg T+O(t^2))d\cL^n\\
&=\int_{B_R \sm K} |\n u|^2 d\cL^n + t\int_{B_R \sm K} |\n u|^2 \dvg T - 2\n u \cdot \n T \n u d\cL^n + O(t^2),
\end{align*}
where the $O(t^2)$ term depends on the derivative of $T$ and on $\int_{B_R \sm K}|\n u|^2\cL^n\leq CR^{n-1}$.

The other integral may be evaluated using the area formula (see, e.g. \cite[Theorem 2.91]{AFP}):
\[
 \sum_i \int_{\phi_t(\partial U_i)\cap B_R}\m_i \theta_i \circ \phi_t^{-1}d\cH^{n-1} 
=\sum_i \int_{\partial U_i \cap B_R} \m_i \theta_i J_{K}(\n \phi_t)d\cH^{n-1},
\]
where the factor $J_{K}(\n \phi_t)$ is the tangential Jacobian of $\phi_t$, and here given by
\[
 J_{K}(\n \phi_t)=J_K (I+t\n T) =  1 + t\dvg\nolimits^K( T) + O(t^2)
\]
at differentiability points of $K$ (see \cite[Theorem 7.31]{AFP} for the computation). Substituting in \eqref{eq:ELopt}, we see that
\[
 0\leq t\left[\int_{B_R \sm K} |\n u|^2 \dvg T - 2\n u \cdot \n T \n u d\cL^n+\sum_i \int_{\partial U_i \cap B_R} 2\m_i \theta_i \dvg\nolimits^K T d\cH^{n-1} \right] +O(t^2).
\]
The first conclusion now follows by sending $t$ to zero from the right or left.

For the second conclusion, select a smooth function $\psi$ supported on $B_r(x)$, and set $f(x)$ to be the signed distance function to $K$ on $B_r(x)$ which is positive on one of the components of $B_r(x)\sm K$ (which we call $V_2$) and negative on the other ($V_1$). Now choose $T=\psi \nabla f$ in \eqref{eq:ELc1}. By recalling that
\[
 \dvg(T|\n u|^2 -2\n u \cdot T \n u)= |\n u|^2 \dvg T - 2 \n u \cdot \n T \n u
\]
for harmonic functions $u$ (Rellich's identity), and that $u$ satisfies a Neumann condition along $K$, we see that
\[
 \int_K \1|\n^{K} u|_{V_1}|^2-|\n^{K} u|_{V_2}|^2 \2 \psi d\cH^{n-1} = - (\m_{V_1}+\m_{V_2})\int_K \dvg\nolimits^K T d\cH^{n-1}.
\]
The term on the right is equal to
\[
 - (\m_{V_1}+\m_{V_2})\int_K \psi H_{V_1} d\cH^{n-1},
\]
and as $\psi$ was arbitrary, the pointwise identity \eqref{eq:ELc2} follows.

Finally, for the last conclusion, take $i=2$. Then consider perturbations like above, but with $\psi \geq 0$ and
\[
 T(x)=\begin{cases}
       \psi(x)\nabla f (x) & x\in V_2 \\
       0 & x\in V_1,K.
      \end{cases}
\]
The pair $(K\cup \phi_t (K),u\circ \phi_t^{-1})$ is an admissible competitor, with the understanding that the complement of the image of $\phi_t$ belongs to $U_0$. By duplicating the computations above, it is easy to check that the claimed inequality follows.
\end{proof}

Let $U$ be a bounded open set and $E$ a set of locally finite perimeter. We say that $E$ is an \emph{inward perimeter minimizer in $U$} if for any set $E'$ with $E \sm E'\cc U$, we have that
\[
 P(E;U)\leq P(E'\cap E;U).
\]
A set is an outward perimeter minimizer if its complement is an inward perimeter minimizer.  If $U=B_R$, it is sufficient to test inward minimality on competitors $E'$ which are smooth away from $\p E$ except on an $\cH^{n-1}$-negligible set (for some representative of $E$). Indeed, it is enough to test against sets attaining the infimum in the minimization problem
\[
 \inf \{P(Q;B_R) : Q \ss E, E \triangle Q \ss B_r \}
\]
for each $r<R$. It is simple to check that the infimum is attained, and that the minimizer $Q$ is a local perimeter minimizer on a small ball $B_r(x)$ around each point of $\p Q$ in $B_\r \sm \p E$. From the regularity of minimal surfaces (see \cite{Giusti}), then, $\p Q$ is smooth at $\cH^{n-1}$-a.e. such point. If $\p Q$ coincides with $\p B_\r \sm \p E$ on a set of positive $\cH^{n-1}$ measure, then it would have to do so at a point of the reduced boundary $z$. Choosing $\r$ small enough that $B_\r(z)\ss B_R \sm \p E$, we have that $Q$ is in fact a perimeter minimizer in $B_\r(Q)$ (this is because $B_r$ is a strict outward minimizer, i.e. has strictly positive mean curvature), which then gives a contradiction with the fact that $\p Q =\p B_r$ locally. This establishes that $\p Q \sm \p E$ is smooth outside a set of $\cH^{n-1}$ measure zero.

\begin{lemma}\label{lem:setmin}
 Let $u$ be a be a minimizer of $F_0$, and $V$ be a connected component of $B_1(x)\sm K$ with $V\ss U_i$ for some $i\geq 1$ and $x\in \p V$. Then for any open $E\subset V$, with $V \sm E \cc B_1(x)$ and $\partial E$ having locally finite $H^{n-1}$ measure and density $1/2$ at $\cH^{n-1}$-a.e. point outside of $\partial V$. Then
 \[
  \cH^{n-1}(\partial E \cap V)\geq \cH^{n-1}(\partial V \sm \partial E) - \frac{1}{\m_i} \int_{V \sm E} |\n u|^2 d\cL^n.
 \]
 In particular, $V$ is an inward perimeter minimizer on $B_1(x)$. If $u$ is constant on $V$, then $V$ is a perimeter minimizer in $B_1(x)$. 
 \end{lemma}

\begin{proof}
Assume $x=0$. Consider the competitor $v=u1_{\R^n \sm (V \sm E)}$; this gives
\[
 \int_{ V \sm  E}|\n u|^2 d\cL^n + \int_{\partial V \sm \partial E}2\theta_i \m_i d\cH^{n-1} \leq \m_i \cH^{n-1}(\partial E \cap V).
\]
This gives the first conclusion. If $u$ is constant on $V$, we show that $V$ is also an outward perimeter minimizer. Indeed, take any $E\supset V$ with $E\sm V\cc B_1$. By the inward minimality of each component of $B_1 \sm K$ not contained in $U_0$, we have that
\[
 P(E;B_1)\geq P(E\cap (U_0 \cup V); B_1),
\]
so it suffices to consider $E$ contained in the closure of $U_0\cup V$. Then take as a competitor
\[
 v(x)=\begin{cases}
    u|_V & E\\
    u(x) & \R^n \sm E.
   \end{cases}
\]
This is admissible, and has the same Dirichlet energy as $u$, so we obtain that
\[
 P(E;B_1)\geq P(V;B_1).
\]
This establishes the outward minimality.
\end{proof}

We pause to explain how the flat-implies-smooth theorem of \cite{Kriv} applies to $F_0$ minimizers. The main observation is that Proposition \ref{prop:fimpsm} may also be applied to $F_0$ minimizers which are obtained as limits of $F$ minimizers, by the following argument: say that for some ball $B_r$, we have that $0\in K$ and
\[
  \sup_{x\in K \cap B_r} d(x,\pi)< \e r,
 \]
Then by the local Hausdorff convergence of $\frac{K_{u_k}-x_k}{s_k}$ to $K$, we will have that for all $k$ large enough, $r s_k\leq r_0$ and
\[
 \sup_{x\in K_{u_k} \cap B_{r s_k}(x_k)} d(x,\pi)< \e r s_k.
\]
Then Proposition \ref{prop:fimpsm} gives that $\frac{K_{u_k}-x_k}{s_k}$ is a union of two $C^{1,\a}$ graphs for each $k$, with uniform constants. The uniform limit of these must then coincide with $K$, and is also a union of two $C^{1,\a}$ graphs. This argument may be iterated to apply to limits of limits of $F_0$ minimizers, and so forth. 

We will say that an $F_0$ minimizer is $\e$-regular if Proposition \ref{prop:fimpsm} applies to it with $r_0=\8$; we have just shown that iterated blow-ups of $F$ minimizers are $\e$-regular. The proof in \cite{Kriv} may be modified to imply that every $F_0$ minimizer is $\e$-regular; however, this is rather technical and irrelevant to our purpose. We will freely restrict to $\e$-regular $F_0$ minimizers whenever convenient to do so. As a final remark, when $n=2$ an argument of Bonnet \cite{Bon} combined with the reduction in \cite[Section 12]{Kriv} gives that under the $\e$-flat assumption, $K$ is locally connected, which leads to a greatly simplified proof that every planar $F_0$ minimizer is $\e$-regular.

The following lemma may be proved in the same way as Lemma \ref{lem:blowup} (we omit the details).

\begin{lemma}\label{lem:secblowup}
 Let $(K_j,u_j)$ be a sequence of $\e$-regular $F_0$ minimizers. Then there is a subsequence along which $K_j$ converge to a closed set $K$ in the local Hausdorff sense. Moreover, on each connected component $V_i$ of $\R^n\sm K $, there is a sequence $c_j$ such that $u-c_j$ converges to a measurable function $u$ in $H^1_{\text{loc}} (V_i)$. For each $V_i$, there is a sequence $U_{i j}$ of connected components of $\R^n\sm K_j$ such that $\cL^n(V_i \sm U_{i j})\rightarrow 0$ and $\mu(U_{i j})\rightarrow \mu_i\in [0,\8)$. Finally, $(K,u)$ is an $F_0$ minimizer, satisfies the uniform density estimates, has at most $N$ of the $V_i$ with nonvanishing $\mu_i$, and
 \[
  \int_{B_R}|\n u|^2 d\cL^n = \lim \int_{B_R} |\n u_j|^2 d\cL^n.
 \]
\end{lemma}

\section[Flatness Criteria in n Dimensions]{Flatness Criteria in $n$ Dimensions}

As in the planar case, we begin by considering $F_0$ minimizers with $N\geq 2$. The analysis is complicated by the possible singularities minimal surfaces exhibit in high dimensions, and relies on the strict maximum principle of Leon Simon \cite{Simon}. We will be able to show that if $N\geq 2$, then $N=2$, $u$ is locally constant, and each of $U_1, U_2$ is a set minimizing finite perimeter. This is not a flatness criterion (consider $K$ a singular minimal cone), but can be used to estimate the dimension of the set of singular points.

\begin{lemma}\label{lem:touching}Let $(K,u)$ be an $\e$-regular $F_0$ minimizer. Assume that $B_1 \sm K$ has at least two connected components $V_1,V_2$ not contained in $U_0$ and that $0\in \p V_1 \cap \p V_2$. Then $N=2$, $U_0=\emptyset$, $u$ is locally constant, and $U_1,U_2$ are complementary perimeter-minimizing sets.
\end{lemma}

\begin{proof}
 Fix $\r<1$, and (unless $u$ is constant on $V_i$) solve the following minimization problems over sets of locally finite perimeter $Q$:
 \[
  P(Q_1;B_1)=\min\{ P(Q; B_1)| Q = V_1 \text{ outside }  B_\r\}
 \]
and
\[
 P(Q_2;B_1)=\min\{ P(Q; B_1)| Q = V_2 \text{ outside }  B_\r\}.
\]
That the minima are attained follows from the compactness theorem for sets of finite perimeter (see \cite{Giusti}). If $u$ is constant on one of the $V_i$, set $Q_i=V_i$, and note that by Lemma \ref{lem:setmin} this set attains the minimum in the above problem. We identify $Q_i$ with the set $Q_i^{(1)}$ of points of Lebesgue density $1$ for $Q_i$, which is open.

We have that $\cL^n(V_i\sm Q_i)=0$. This is by construction if $u$ is constant, while if $u$ is nonconstant we may use $V_i\sm Q_i$ as the set $E$ in Lemma \ref{lem:setmin} to obtain
\[
 P(V_i;B_1)\leq P(Q_i;B_1) - c \int_{V_i\sm Q_i}|\n u|^2 d\cL^n.
\]
As $u$ is nonconstant, the integral on the right is strictly positive, and this strict inequality contradicts the minimality of $Q_i$. A similar argument gives that $\cL^n(Q_1\sm V_2)=0$ and vice versa.

We also have that the $Q_i$ are connected. By possibly replacing $Q_1$ by $Q_1 \sm \bar{Q}_2$ and $Q_2$ by $Q_2\sm \bar{Q}_1$ (this does not increase their perimeter, hence they still attain the minima above), we may ensure that $Q_1\cap Q_2 =\emptyset$. On the other hand, $0\in \p V_1 \cap \p V_2$, and so $0\in \p Q_1 \cap \p Q_2$. Applying the maximum principle of \cite{Simon} gives that $B_\r \cap \p Q_1= B_\r \cap \p Q_2$.

We claim that this implies that $S=\p B_\r \sm (\bar{V}_1\cup\bar{V}_2)$ is empty. Indeed, take $z\in S$ and a ball with $B_\s (z)\cap \p B_\r \ss S$  Possibly choosing a smaller value of $\s$, we may arrange to have $B_\s (y)\cap K = \emptyset$. Now note that $B_\r$ itself is an outward perimeter minimizer, and hence $Q_i$ is a local perimeter minimizer in $B_\s(z)$: for any set of finite perimeter $E$ with $Q_i \triangle E \cc B_\s(z)$,
\[
 P(B_\r \cap E ; B_\s(z))\leq P(E;B_\s(z)),
\]
while
\[
 P(Q_i;B_\s(z))\leq P(B_\r \cap E ; B_\s(z))
\]
by construction (as $B_s(z)\cap V_i =\emptyset$). However, $\p B_\r \cap  B_\s(z) \ss \p Q_1 \cup \p Q_2$ from above. If in fact $\p B_\r \cap B_\s(z) \ss \p Q_i$ for one of the $i$, we immediately obtain a contradiction with the fact that $\p Q_i$ coincides with $\p B_\r$ at a regular point, and hence has strictly positive curvature. On the other hand, if $\p Q_1 \cap \p Q_2$ intersect in $B_\s (z)$, we may again apply the strict maximum principle to obtain that $\bar{Q}_1 \cup \bar{Q}_2$ exhaust $B_\s(z)$, which contradicts that they are both contained in $B_\r$.

By repeating this argument for each value of $\r$, we see that $V_1\cup V_2 \cup K$ partitions $B_1$, Applying Lemma \ref{lem:setmin}, each $V_i$ is a perimeter minimizer (being an inward minimizer and the complement of an outward minimizer). Now take a regular point $x\in \p V_1$ and apply Lemma \ref{lem:EL} on a neighborhood of this point (where, by the regularity of minimal surfaces, see \cite{Giusti}, $\p V_1$ is given by an analytic graph); as the mean curvature of $\p V_1$ is zero, this implies that $\n u$ vanishes along $\p V_1$ on this neighborhood. By the unique continuation property of harmonic functions, this implies $U$ is constant on $V_1$. Similarly, $u$ is constant on $V_2$. Finally, this argument may be applied to larger and larger balls to obtain the conclusion of the lemma (noting that $V_1$ and $V_2$ may not be contained in the same connected component of $\R^n\sm K$, as then $(K\sm B_1,u)$ would be an admissible competitor which strictly reduces $F_0$).
\end{proof}

To simplify notation, we say that a point $x\in k$ is a \emph{smooth point} if the conclusion of Proposition \ref{prop:fimpsm} applies to some ball $B_r(x)$.

\begin{theorem}\label{thm:ndimsm} Let $(K,u)$ be a local $F$ minimizer on $\W$.  Set
\[
\Sigma := \{x\in \W\cap K: \limsup_{r\searrow}\frac{1}{r^{n-1}}\int_{B_r(x)}|\n u|^2 d\cL^n =0\}.
\]
Then $\Sigma$ contains any point at which one of the following holds:
\begin{enumerate}
 \item $\liminf_{r\searrow 0} \frac{\cL^n (\{u=0\}\cap B_r(x))}{r^n} =0$
 \item  for some $r$, $B_r(x)\sm K$ has two connected components, each of which has with $x$ in the boundary and $u>0$ on its interior
\end{enumerate}
Furthermore, there is a relatively closed subset $\Sigma^*\ss \Sigma$ with $\cH^\t ( \Sigma^*)=0$ for each $\t>n-8$ such that $\Sigma\sm \Sigma^*$ contains only smooth points.
\end{theorem}

\begin{proof} 
 Let $x\in \Sigma$ and set $E$ to be the set of nonsmooth points of $K\cap \W$. We show that the $\cH^\t$ measure of $E \cap \Sigma$ is zero for each $\t>n-8$.

Indeed, this is a consequence of the classical dimension reduction argument of Federer; we briefly summarize it here, referring to \cite{Giusti} for most of the proofs. Assume that this is not the case, so in other words $\cH^\t(E \cap \Sigma)>0$.  By \cite[Proposition 11.3]{Giusti}, we have that
\[
 \limsup \frac{\cH^{\t}_\8(B_\s (y)\cap E }{\s^\t}> c(n,\t)>0
\]
for $\cH^\t-$a.e. $y\in E\cap \Sigma$. Fix a $y$ for which the above holds, and a subsequence with  
\[
 \lim_k \frac{\cH^{\t}_\8(B_{\s_k} (y)\cap E }{\s_k^\t}> c(n,\t)>0.
\]
Passing to a further subsequence, we may take 
\[ 
\1\frac{K-y}{s_k},\frac{u(s_k(\cdot -y))-c_k}{\sqrt{s_k}} \2 \rightarrow (K_\8,u_\8)
\]
in the sense of Lemma \ref{lem:blowup}. As $y\in \Sigma$< $u_\8$ is locally constant, and hence by Lemma \ref{lem:setmin} $K_\8$ is a finite union of minimal hypersurfaces. By (\cite[Lemma 11.5]{Giusti}, which applies to Hausdorff convergence of $\e$-regular sets), we have that
\[
 \cH^\t_\8 {E_\8 \cap B_1}\geq c(n,\t)>0,
\]
where $E_\8$ is the singular set of $K_\8$. It is a classical result of minimal surface theory, however \cite[Theorem 11.8]{Giusti}, that $\cH^\t_\8 (E_\8)=0$, which is a contradiction.

Now we show that criteria (1) and (2) ensure that a point $x$ lies in $\Sigma$. For (1), take any sequence $\s_k\rightarrow 0$ and extract a subsequence with
\[ 
\1\frac{K-x}{s_k},\frac{u(s_k(\cdot -x))-c_k}{\sqrt{s_k}} \2 \rightarrow (K_\8,u_\8)
\]
in the sense of Lemma \ref{lem:blowup}. By (1), we have in addition that the limiting zero component $U_0$ associated to $(K_\8,u_\8)$ is empty, and $0\in K$. Thus $K$ has nontrivial $\cH^{n-1}$ measure, and hence there is a $y\in K_\8$ which is a smooth point. Apply Lemma \ref{lem:touching} in a small neighborhood of $y$, noting that as $U_0$ is empty and by the $\e$-regularity property $K$ locally disconnects $B_r(y)$, the hypotheses are satisfied. It follows that $u_\8$ is locally constant, and applying part (6) of Lemma \ref{lem:blowup}, we have
\[
 \lim_k \frac{1}{\s_k^{n-1}}\int_{B_{\s_k}(x)}|\n u|^2 d\cL^n = 0.
\]
After applying argument along each subsequence, we deduce that $x \in \Sigma$.

If (2) holds, proceed in the same manner, but now note that by the positive Lebesgue density of each connected component near $x$, the limit will have at least two components $U_1$ and $U_2$ which have $0$ in their boundary. Thus Lemma \ref{lem:touching} may be applied to $B_1(0)$.
\end{proof}

Due to the smoothness of minimal surfaces in dimensions below 8, we have the following stronger result, which establishes an energy gap between regular and singular points.

\begin{theorem}\label{thm:energygap} Let $(K,u)$ be an $\e$-regular $F_0$ minimizer with $0\in K$ and $n\leq 7$. Then either (a) $K$ is a hyperplane or a union of two parallel hyperplanes, or (b),
\[
 \liminf_{R\rightarrow \8} \frac{1}{R^{n-1}}\int_{B_R}|\n u|^2 d\cL^n \geq \z>0.
\]
Let $u$ be a local $F$ minimizer on $\W$ with $0\in K$ and $n\leq 7$. Then either $0$ is a smooth point or
\[
 \liminf_{r\rightarrow 0} \frac{1}{r^{n-1}}\int_{B_r}|\n u|^2 d\cL^n \geq \zeta>0.
\]
The constant $\z$ depends only on $n$ and the $\e$-regularity constants.
Let $(K,u)$ be an $\e$-regular $F_0$ minimizer with $0\in K$ and $n\leq 7$, and assume $N\geq 2$. Then $N=2$, $u$ is locally constant, and $K$ is of type (a).
\end{theorem}

The proof will follow easily from the following lemma, which is useful in its own right.

\begin{lemma}\label{lem:encrit}Let $(K,u)$ be an $\e$-regular $F_0$ minimizer with $0\in K$ and $n\leq 7$. Then there is a universal constant $\z$ (depending only on $n$ and the $\e$-regularity constants) such that if
 \[
  \int_{B_1}|\n u|^2 d\cL^n \leq \z,
 \]
 then $0$ is a smooth point of $K$, and moreover $K\cap B_{1/2}$ is given by a union of two $C^{1,\a}$ graphs lying $\e$-close in $C^{1,\a}$ topology to a pair of parallel lines, one of which passes through $0$.
\end{lemma}

\begin{proof}
We first show that (for a small enough $\z$) $0$ is a smooth point. If this were not the case, we could find a sequence of $\e$-regular $F_0$ minimizers  $(K_k,u_k)\rightarrow (K_\8,u_\8)$ in the sense of Lemma \ref{lem:secblowup}, with $0\in K_\8$ not a smooth point and $u_\8$ constant on $B_1$. By Lemma \ref{lem:setmin}, $K_\8\cap B_1$ is a union of minimal surfaces, and hence (by the regularity of minimal surfaces) $0$ is a smooth point of $K_\8$.  This contradicts $0$ not being a smooth point of any $K_k$. 

 We now show the other property, arguing by contradiction and proceeding in the same manner to extract a convergent subsequence. Note that each component $U_i$ associated with $(K_\8,u_\8)$. intersects nontrivially with $B_1$ is an entire perimeter minimizer by Lemma \ref{lem:setmin}, and so a half-space by Bernstein's theorem. There are now two possibilities: either only one component $U_i$ intersects $B_1$, or at least two do. In the first case, for $k$ large enough, the $\e$-flatness hypothesis for $(K_k,u_k)$ is satisfied on $B_1$, and hence $(K_k,u_k)$ is given by a pair of $C^{1,\a}$ graphs on $B_\eta$ for some $\eta>0$ depending only on the $\e$-regularity constants (indeed, in this case a single graph would suffice). If at least two components $U_1, U_2$ intersect $B_1$, then it follows that either $K$ is a line and $U_1,U_2$ are complementary (in which case the same argument as before applies, now requiring two graphs) or $K$ is a union of two parallel lines, one passing through $0$. Choose $r$ smaller than the distance between these two lines, and observe that for $k$ large enough and any $x\in K_k\cap B_1$, $B_r(x)$ is $\e$-flat, and so $K\cap B_r(x)$ is given by a $C^{1,\a}$ graph $\e$-close to one of the two lines. Together with the fact that for $k$ large, $K_k$ is contained in an $r$-neighborhood of $K_\8$, this leads to a contradiction. 
\end{proof}

\begin{proof}[Proof of Theorem.]
The second conclusion follows directly from applying Lemma \ref{lem:encrit} to a blow-up sequence with
\[
 \frac{1}{R_k^{n-1}}\int_{B_{R_k}}|\n u|^2 d\cL^n \rightarrow 0.
\]

Now we show the first conclusion. Assume that $(K,u)$ is not of type (b). Then there is a sequence $R_k\rightarrow \8$ such that
\[
 \frac{1}{R_k^{n-1}}\int_{B_{R_k}} |\n u|^2 d\cL^n \leq \z
\]
for each $k$. We may then apply Lemma \ref{lem:encrit} to $K/R_k$ to obtain that $K\cap B_{\eta R_k}$ is a union of two $C^{1,\a}$ graphs in some coordinate system $\R^{n-1}\times\R$, and moreover
\[
  d(K, \pi_1 \cup \pi_2;B_1) \leq CR_k^{-\a},
\]
where $\pi_1,\pi_2$ are the tangent hyperplanes to $K$ at $0$ and at the unique other point $x\in K \cap \{0\}\times \R$. Taking the limit in $k$ gives that $K \cap B_1$ coincides with a pair of hyperplanes. Repeating for dilates of $(K,u)$ gives that $K$ is of type (a).

We now show the third conclusion. If at least two of the components $U_i$, $i\geq 1$ share a boundary point, Lemma \ref{lem:touching} allows us to conclude. If not, consider the blow-down limit
\[
 \1 \frac{K}{R_k},\frac{u(R_k \cdot)-c_k}{\sqrt{R_k}} \2 \rightarrow (K_\8,u_\8)
\]
along some subsequence in the sense of Lemma \ref{lem:secblowup}. It is simple to check that $N\geq 2$ for $(K_\8,u_\8)$, and that $0$ lies in the boundary of at least two of the components $U_i$, $i\geq 2$. Thus $K_\8$ is a line, and $u_\8$ is constant. It follows that
\[
 \frac{1}{R_k^{n-1}}\int_{B_{R_k}} |\n u|^2 d\cL^n \rightarrow 0,
\]
and so $K$ is not of type (b). Thus $K$ is of type (a), and the conclusion follows.
\end{proof}

\section{Flatness Criteria in the Plane}

In the plane, the structure of $F_0$ minimizers is particularly simple, and we present a criterion for smoothness that applies at points with positive density of the set $\{u=0\}$ We suspect that every point in $K$ is a smooth point when $n=2$, but this is left open. We also point out that in this case in particular, the global problem has a strong resemblance to the famous Mumford-Shah problem, and many of the partial results on ``global Bonnet minimizers'' can be carried over to our setting (for example the formulas of Bonnet \cite{Bon} and Leger \cite{L}, and some of the theorems of David and Leger \cite{DL}, see also \cite{D}). The difference between the problem treated here and in the Mumford-Shah case is that we count perimeter with a multiplicity determined by the Lebesgue density of the set it is bounding, whereas the Mumford-Shah functional counts it with constant multiplicity one. As a consequence, while the Mumford-Shah functional forms cracks and propellers, ours only forms ``holes.''

We only consider the case $N=1$, as the other is already treated in the previous section.

\begin{lemma}\label{lem:convex} Let $(K,u)$ be an $\e$-regular minimizer of $F_0$ with $K$ not equal to a line, and $n=2$, $N=1$. Then
\begin{enumerate}
 \item $U_0$ is composed of a disjoint union of open, convex sets.
 \item At every point in $K$, $\theta_1<1$.
 \item Each component $V$ of $U_0$ is bounded, and has as boundary a $C^\8$ curve.
 \item There is a universal constant $c_0$ such that
 \[
  c_0 P(V)\leq d(\bar{V}, K\sm \p V) \leq c_0^{-1} P(V),
 \]
and such that for any line $\pi$, the projection $\pi(V)$ of $V$ onto $\pi$ satisfies
\[
 \cL^1 (\pi(V))\geq c_0 P(V).
\]
\end{enumerate}
\end{lemma}

The fourth property says that each ``hole'' $V$ is separated from the rest of $K$ in a scale-invariant way, and furthermore is relatively round. An alternative way of describing the last inequality is that there is an ellipse $Q$ such that $Q\ss V$ and $V$ is contained in its double (this is true for all convex sets, by the lemma of John \cite{John}), and we claim that the eccentricity of the ellipse is bounded away from $1$.

The third property implies that the union of the boundaries of the components of $U_0$, the sets $V_i$, contains only smooth points. The question then remains as to whether $K\sm \cup_i \p V_i$ is empty or not. We effectively show that any singular point is an accumulation point of holes of decreasing size.

\begin{proof}
 We will first show that every component $V$ of $U_0$ is convex. Indeed, take any  two points $x,u\in V$, and consider the line segment $l$ joining them. We aim to prove that if $U_1 \cap l$ is nonempty, then $U_1 \sm l$ has a bounded connected component. If that is the case, then an application of Lemma \ref{lem:setmin} yields a contradiction.
 
 To show this, we must eliminate several possibilities. First, say $U_1 \sm l$ is connected. This contradicts the fact that $V$ was assumed connected. Indeed, if $z\in l \cap U_1$, then there is a small ball $B_r(z)\ss U_1$ separated into two half-balls by $l$, and points $w_1,w_2$ in each half-ball. Now take a (simple) curve $\gamma: [0.T]\rightarrow U_1 \sm l$ connecting $w_1$ and $w_2$, and extend it to a loop by adjoining the line segment connecting $w_1$ to $w_2$. Possibly modifying $\gamma$ within $B_r(z)$, this loop may be taken to be simple, and so divides $\R^2$ into two components. As the loop only intersects $l$ at one point, $x$ and $y$ each lie in different components; it follows (as the loop stays in the complement of $V$) that these components form a nontrivial separation of $V$.
 
 Next, we rule out the case that $U_1 \sm l$ is disconnected, but all components are unbounded. In this case, take $\gamma : [0,1] \rightarrow V$ to be a simple curve connecting $x$ and $y$. Let $z$ and $B_r(z)$ be as before, and note that $B_r(z)\cap l$ does not intersect $\gamma$. Let $l'\ss l$ be the line segment contained in $l$, containing $z$, and intersecting $\gamma$ only at the endpoints (say $\gamma(t_1)$ and $\g(t_2)$. Then the union $\g'=l'\cup \g |_{[t_1,t_2]}$ is a simple loop, and separates the plane into two components, one of which is bounded. As $\g' \ss \R^2 \sm (U_1 \sm l)$, each of these components contains (at least) one nontrivial component of ($U_1 \sm l$), so one of those components must be bounded. This establishes $(1)$.
 
 The property $(2)$ follows from the regularity criterion (1) in Theorem \ref{thm:ndimsm}, which implies that any point of density 1 is smooth, together with Lemma \ref{lem:touching}, which then says that $N=2$ and $k$ is a line. This contradicts the assumption that $N=1$.

 We now show each component $V$ of $U_0$ is bounded. Indeed, assume $0\in \p V$ and $V$ is unbounded. We claim (entirely by virtue of $V$ being convex) that the blow-downs $\p V /R$ converge in local Hausdorff topology to the boundary of a convex cone (or to a ray) as $R\rightarrow \8$. To see this, after choosing an appropriate coordinate system write $\p V = \{ (x, g(x))\in \R^2 | x\in \R\}$ for some convex function $g:\R\rightarrow \R$ with $g(0)=0$ and $g\geq 0$. Then by virtue of convexity,
 \[
  \frac{g(x)}{x}
 \]
is an increasing function of $x$, and so has a limit $g_\pm$ (possibly infinite) as $x\rightarrow \pm \8$. It is now straightforward to show that $\p V$ converges to the cone containing the rays $\{(x,g_+ (x)): x\geq 0\}$ if $g_+$ is finite, $\{(x,g_- (x)): x\leq 0\}$ if $g_-$ is finite, and $\{(0,y):y\geq 0\}$ if one of them is infinite. 

Consider then a subsequence $R_k$ such that
\[
 \1\frac{K}{R_k}, \frac{u(R_k \cdot )-c_k}{\sqrt{R_k}} \2 \rightarrow (K_\8,u_\8)
\]
in the sense of Lemma \ref{lem:secblowup}. We will show that $u_\8$ is locally constant. If $\R^2 \sm K_\8$ has at least two connected components not in $(U_0)_\8$, then this follows from Theorem \ref{thm:energygap}. If not, we know that $K_\8$ contains at least one ray, and hence a smooth point $z$ on this ray. In a neighborhood of this smooth point, we know that $\theta_i<1$, and so must be $\frac{1}{2}$, meaning that $K_\8 \cap B_\r (z)$ is a single line segment. Applying Lemma \ref{lem:EL}, we have that $u$ is locally constant on $B_\r \cap K_\8$, and so on the entire component $U_1$.

Now, from Lemma \ref{lem:secblowup},
\[
 \frac{1}{R_k}\int_{B_{R_k}}|\n u|^2 d\cL^n \rightarrow 0,
\]
and so by Theorem \ref{thm:energygap} we must have $K$ is a line; this is a contradiction.

Let us now fix $V$ a component of $U_0$ with $0\in \p V$, and study the blow-up
\[
 \1\frac{K}{r_k}, \frac{u(r_k \cdot )-c_k}{\sqrt{r_k}} \2 \rightarrow (K_\8,u_\8)
\]
along some sequence $r_k \rightarrow 0$. By an argument analogous to the one for blow-downs above, $\p V/r_k$ converges to its (unique) tangent cone containing either one or two rays. As argued previously, this implies that $u_\8$ is locally constant, and this gives that
\[
 \frac{1}{r_k}\int_{B_{r_k}}|\n u|^2 d\cL^n \rightarrow 0.
\]
Applying Lemma \ref{lem:encrit}, this means that $0$ is a smooth point of $K$. This completes the proof of $(3)$.

Assume that the uniform projection property in (4) fails. After rescaling and rotating, this means there is a sequence $(K_k,u_k)$ of $F_0$ minimizers with components $V_k$ of $(U_0)_k$ such that $0\in \p V_k$, $\text{diam}(V_k) =1$, and
\[
 V_k \ss \{x>0\} \times \{-\frac{1}{k}< y <\frac{1}{k}\}.
\]
Extracting a subsequence converging to $(K_\8,u_\8)$, we see that $K_\8$ must contain the line segment $\{0\}\times [0,1]\}$. This implies that $K_\8$ is a line or union of two lines and $u_\8$ is locally constant. In particular, for $k$ large enough, we have that
\[
 \int_{B_{200}}|\n u_k|^2 d\cL^n \leq \z
\]
for $\z$ arbitrarily small. By Lemma \ref{lem:encrit}, we have that $K_k\cap B_{100}$ is given by a union of two $C^{1,\alpha}$ graphs. On the other hand, around $V_k$ there is a small neighborhood in which $K$ is a smooth Jordan curve, and this may not be expressed as the union of two $C^{1,\a}$ graphs.

A consequence of this is the fact that if $V$ is a component of $U_0$ with $0\in \p V$ and $\text{diam}(V)=1$, then $\cL^2 (B_{1/2}\cap V)\geq c>0$. As for a convex set the quantity
\[
 \frac{\cL^2 (V \cap B_\r)}{\r^2}
\]
is decreasing in $\r$, we have the uniform density estimate
\begin{equation}\label{eq:uniformzden}
 \frac{\cL^2 (V \cap B_\r)}{\r^2}\geq c_0\qquad \forall \r\in (0,\frac{1}{2}).
\end{equation}

We now show the following uniform flatness property: let $V$ be a component of $U_0$ with $\text{diam}{V}=1$ and $0\in \p V$; then (after a rotation) for some universal $q>0$, $B_q \cap K \ss \R \times \{|y|\leq q \e\}$. Notice that by the smoothness of each component $V$, we have that
\[
 q(V);=\max \3 q| B_q \cap K \ss \R \times \{|y|\leq q \e\} \text{ after rotation} \4>0.
\]
We claim that $q(V)$ is bounded from below by a universal constant. Assume this is not the case; then there is a sequence $(K_k,u_k)$ and components $V_k$ of $(U_0)_k$ such that $0\in \p V_k$, $\text{diam}(V_k)=1$, and $q(V_k)\rightarrow 0$. We will assume that the line $\{y=0\}$ is the tangent line to each $V_k$ at $0$, and that $V_k\ss \{y>0\}$. Then consider the dilations
\[
 \1 K_k/q(V_k), \frac{u_k(q(V_k)\cdot)-c_k}{\sqrt{q_k}} \2\rightarrow (K_\8,u_\8),
\]
where the convergence is as in Lemma \ref{lem:secblowup} after passing to a subsequence. Then as $0\in K_k$ for each $k$, $0\in K_\8$. By the maximality of $q(V_k)$ and the convexity of each $V_k$, one of the two points $\p B_{q(V_K)} \cap \{y = q(V_k)\e\}$ lies in $K_k$; it follows that one of the two points in $\p B_1 \cap \{y=\e\}$ lies in $K_\8$. The convex sets $V_k/q(V_k)$ converge in $L^1_{\text{loc}}$ to an unbounded convex component $V$ of $(U_0)_\8$, with $V\ss \{y>0\}$. By the uniform density property \eqref{eq:uniformzden}, $\cL^2 (V_k/q(V_k) \cap B_\r)\geq c_0 \r^2$ for each $\r\leq 1$, so  $V$ is nonempty and $0\in \p V$. By property (3), it follows that $K_\8$ must be a line or pair of parallel lines, so in particular $V=\{y>0\}$. This contradicts the fact that one of $\p B_1 \cap \{y=\e\}$ lies in $K_\8$.

The uniform separation property in (4) follows immediately from this uniform flatness property.
 \end{proof}

We offer the following corollary, which gives a better description of the singular set of an $F$ minimizer. 

\begin{corollary}
 Let $u$ be a local $F$ minimizer on $\W$. Then for each component $V$ of $\{u=0\}\cap B_r(x)$ with $B_r(x)\cc \W$, we have that $\p V \cap B_\r(x)$ is a finite union of $C^{1,\a}$ arcs for $\r<r$. The arcs may meet only pairwise, and only at their endpoints, where they are tangent to each other. The arcs are $C^\8$ on their interiors.
\end{corollary}

\begin{proof}
 We first show that if $u$ is as stated and $0\in \p V\cap B_r(x)$ with $V$ a component of $\{u=0\}\cap B_r(x)$, then $0$ is a smooth point for $K$. To this end, take any sequence of blow-ups
 \[
  \1 \frac{K}{r_k},\frac{u(r_k \cdot)-c_k}{\sqrt{r_k}} \2 \rightarrow (K_\8,u_\8)
 \]
in the sense of Lemma \ref{lem:blowup}. In particular, the sets $(\p V\cap B_r(x)) /r_k$ converge in the local Hausdorff sense to a closed set $K_0\ss K_\8$. We now use the following basic property of Hausdorff convergence: if compact sets $J_k\rightarrow J$ in Hausdorff topology, and $J_k$ are connected, then so is $J$. Applying to $\bar{B}_{t r_k} \cap \p V$ (which has at most $N$  connected components; this follows from the fact that $\{u>0\}$ has at most finitely many local connected components), we have that $K_0\cap \bar{B}_t$ has at most $N$ connected components for each $t$. This implies that $K_0$, and so in particular $K_\8$, contains an unbounded connected component. However, by Lemma \ref{lem:convex}, either this means $K_\8$ is a line or a pair of lines (otherwise every connected component of $K$ is either a smooth Jordan curve, which is bounded, or has zero $\cH^1$ measure, which implies it is also bounded). This means $u_\8$ is locally constant, and so by Theorem \ref{thm:energygap} $0$ is a smooth point.

Now each point $y\in \p V \cap B_r(x)$ is of one of two types: either the Lebesgue density of $\{u=0\}$ at $y$ is positive (we call the set of such points $E_1$) or it is zero (these are in $E_2$). Note that $E_1$ is relatively open, and in a neighborhood of each $x\in E_1$, $K$ is given by a single $C^\8$ graph. The proof of the corollary will be complete once we show that each point in $E_2$ is isolated.

Take a point (say $0$) in $E_2$. We have that on a small ball $B_\s (0)$, $K$ is given (in some coordinates) by the graphs of a pair of $C^{1,\a}$ functions $g_-,g_+$, with $g_-\leq g_+$ and $g_-(0)=g_+(0)=0$. Furthermore, $\{u=0\}\cap B_\s \ss \R \times \{g_-<g_+\}$. As $0\in \p V$, we must have that the graphs of $g_-,g_+$ do not touch on $B_\s \cap \{(a,b):a>0\}$ (up to a rotation of $180$ degrees). In particular, this means that $\p V\cap B_\s \cap \{a>0\}\ss E_1$.  Now say there is a point $z\in V$ in $\{a<0\}\cap B_\s$. Then we may find a simple curve $\g \ss V$ joining $z$ to some point $w\in V \cap \{a>0\}\cap B_\s$. Modify $\g$ as needed to ensure that $\g \cap B_\s \cap \{z_1\leq a\leq \w_1\}=\emptyset$. We may let
\[
 \g' = \g \cup \{(z_1,t)|t\in [z_2,g_+(z_1)]\}\cup \{(w_1,t)|t\in [w_2,g_+(w_1)]\}\cup \{(t,g_+(t))|t\in [z_1,w_1]\},
\]
which is a Jordan curve. Then one of the two sets $B_\s \cap \{(a,b):b>g_+(a)\}$ and $B_\s \cap \{(a,b):b<g_-(a)\}$ is contained in the interior of $\g'$, and so the component of $\{u>0\}$ containing it must also be bounded. This is a contradiction.

We have now shown that $B_\s \cap \p V \ss E_1$, which completes the argument.
\end{proof}

\section*{Acknowledgments}Both authors were partially supported by NSF grant DMS-1065926. DK was also partially supported by the NSF MSPRF fellowship DMS-1502852. LC was also partially supported by NSF DMS-1540162.

\bibliographystyle{plain}
\bibliography{Paper1}

\end{document}